\documentclass[9pt,shortpaper,twoside,web]{ieeecolor}
\usepackage{generic}
\usepackage{cite}
\usepackage{amsmath,amssymb,amsfonts}
\usepackage{algorithmic}
\usepackage{graphicx,color}

\usepackage{textcomp}
\usepackage{enumerate}
\usepackage{indentfirst}
\usepackage{graphicx} % Include this line if your
\usepackage[dvips]{epsfig}
\usepackage{tabularx}% Latex-Dvi-Pdf hyperlink
\usepackage{graphicx} % Include this line if your
\usepackage[dvips]{epsfig} % or this line, depending on which
\usepackage{wrapfig} % you prefer.
\usepackage{epstopdf} % or this line, depending on which
\usepackage{wrapfig} % you prefer.
\usepackage{subfigure}
\usepackage{epstopdf}
\usepackage{amsmath}
\usepackage{amsfonts}
\usepackage{graphicx}
\usepackage{amssymb}
\usepackage{epsfig}
\usepackage{algorithm}
\usepackage{amsmath}
\usepackage{amsfonts}
\usepackage{graphicx}
\usepackage{amssymb}
\usepackage{epsfig}
\usepackage{color}

\newtheorem{condition**}{A*}
\newtheorem{condition***}{C*}
\newtheorem{condition*}{C}

\newtheorem{proposition}{Proposition}[section]

\newtheorem{definition}{Definition}[section]
\newtheorem{theorem}{Theorem}[section]

\newtheorem{assumption}{Assumption}[section]

\renewcommand{\vec}{\mathrm{vec}}
\def\QEDclosed{\mbox{\rule[0pt]{1.3ex}{1.3ex}}} % ¶¨ÒåÊµÐÄ·û
\def\QEDopen{{\setlength{\fboxsep}{0pt}\setlength{\fboxrule}{0.2pt}\fbox{\rule[0pt]{0pt}{1.3ex}\rule[0pt]{1.3ex}{0pt}}}} %¶¨Òå¿ÕÐÄ·û
\def\QED{\QEDopen} % Ñ¡Ìî\QEDclosedµÃµ½ÊµÐÄ
\newcommand{\var}{\mathrm{Var}}

\newcommand{\cD}{\mathcal D}
\newcommand{\cS}{\mathcal S}
\newcommand{\cF}{\mathcal F}
\newcommand{\cU}{\mathcal U}
\newcommand{\cX}{\mathcal X}
\newcommand{\cJ}{\mathcal J}
\newcommand{\cT}{\mathcal T}
\newcommand{\cK}{\mathcal K}
\newcommand{\cI}{\mathcal I}
\newcommand{\cR}{\mathcal R}

\newcommand{\bbF}{\mathbb F}
\newcommand{\bbR}{\mathbb R}
\newcommand{\bbP}{\mathbb P}
\newcommand{\bbN}{\mathbb N}
\newcommand{\E}{\mathbb E}

\newcommand{\barA}{\overline{A}}
\newcommand{\barB}{\overline{B}}
\newcommand{\barC}{\overline{C}}
\newcommand{\barD}{\overline{D}}
\newcommand{\barQ}{\overline{Q}}
\newcommand{\barR}{\overline{R}}
\newcommand{\barS}{\overline{S}}

\newcommand{\bfQ}{\mathbf Q}
\newcommand{\bfR}{\mathbf R}
\newcommand{\bfS}{\mathbf S}

\newcommand{\hatA}{\widehat A}
\newcommand{\hatB}{\widehat B}
\newcommand{\hatC}{\widehat C}
\newcommand{\hatD}{\widehat D}
\newcommand{\hatR}{\widehat R}
\newcommand{\hatS}{\widehat S} 
\newcommand{\hatQ}{\widehat Q} 

\newcommand{\hatK}{\widehat K}
\newcommand{\hatP}{\widehat P} 
 
\newcommand{\loc}{\mathrm{loc}} 

\newcommand{\dd}{\ensuremath{\operatorname{d}\!}}

\newcommand{\ds}{\ensuremath{\operatorname{d}\! s}}

\newcommand{\dw}{\ensuremath{\operatorname{d}\! W}}

\def\BibTeX{{\rm B\kern-.05em{\sc i\kern-.025em b}\kern-.08em
 T\kern-.1667em\lower.7ex\hbox{E}\kern-.125emX}}
\begin{document}

\title{Policy Iteration Reinforcement Learning Method for Continuous-Time Linear-Quadratic Mean-Field Control Problems}

\author{Na Li, Xun Li, and Zuo Quan Xu
\thanks{N. Li acknowledges the financial support from the NSFC (No.~12171279, No.~11801317), the Consulting and Research Project of China Engineering Science and Technology Development Strategy Shandong Research Institute (No.~202302SDZD04), and the International Cooperation Research Platform of Shandong University of Finance and Economics. X. Li acknowledges the financial support from the Research Grants Council of Hong Kong (under grants No.~15216720, No.~15221621, No.~15226922), and PolyU 1-ZVXA, 4-ZZLT and 4-ZZP4. Z. Q. Xu acknowledges the financial support from the NSFC (No. 11971409), Hong Kong RGC (GRF 15202421, 15204622 and 15203423), The PolyU-SDU Joint Research Center on Financial Mathematics, The CAS AMSS-PolyU Joint Laboratory of Applied Mathematics, The Research Centre for Quantitative Finance (1-CE03), and internal grants from The Hong Kong Polytechnic University. Corresponding author: Na Li.}
\thanks{N. Li, School of Statistics and Mathematics,
Shandong University of Finance and Economics, Jinan, Shandong, 250014, China. (e-mail: naibor@163.com)}
\thanks{X. Li, Department of Applied Mathematics, The Hong Kong Polytechnic University, Hong Kong, China. (e-mail: li.xun@polyu.edu.hk)}
\thanks{Z. Q. Xu, Department of Applied Mathematics, The Hong Kong Polytechnic University, Hong Kong, China. (e-mail: maxu@polyu.edu.hk)}}

\maketitle

\begin{abstract}
This paper employs a policy iteration reinforcement learning (RL) method to study continuous-time linear-quadratic mean-field control problems in infinite horizon. The drift and diffusion terms in the dynamics involve the states, the controls, and their conditional expectations. We investigate the stabilizability and convergence of the RL algorithm using a Lyapunov Recursion. Instead of solving a pair of coupled Riccati equations, the RL technique focuses on strengthening an auxiliary function and the cost functional as the objective functions and updating the new policy to compute the optimal control via state trajectories. A numerical example sheds light on the established theoretical results. 
\end{abstract}

\section{Introduction}
The mean-field (MF) problems have important applications in various fields, including, but not limited to, science, engineering, financial management, and economics. 
Since the independent introduction by Lasry and Lions \cite{Lasry-Lions} and Huang {\em et al.} \cite{Huang-Caines-Malhame-2007}, there has been increasing interest in the studying of MF problems as well as addressing their applications. As pointed out by Bensoussan {\em et al.} \cite{Bensoussan-Frehse-Yam-2013}, MF games (MFG) and MF control (MFC) bring new problems in control theory. In MFGs, the MF term is considered an external given so that the agent does not influence it. Because of this, MFGs can be tackled by first solving a standard control problem and then finding an equilibrium. In MFC problems, by contrast, the agent can influence the MF term so that they are not standard control problems in the form of \cite{Yong-Zhou-1999}. As stated in Yong \cite{Yong-2017}, people might like to have the optimal control and state to be not too ``random". To achieve that, one could include a variation of the state process and/or variation of the control process in the cost functional.
As an important class of optimal control problems, linear quadratic (LQ) problems with MF terms have attracted extensive research attention. Yong \cite{Yong-2013} presented the feedback representation for the optimal control of deterministic coefficient MFC-LQ problems by two Riccati differential equations, which are uniquely solvable under certain conditions. Further, Yong \cite{Yong-2017} investigated the time-inconsistent feature of MFC-LQ problems by giving pre-commitment and time-consistent solutions. Moreover, in the infinite horizon, Huang {\em et al.} \cite{Huang-Li-Yong-2015} studied several different stabilizabilities of MFC-LQ problems. They solved a kind of {\em generalized algebraic Riccati equations} (GAREs) by the semi-definite programming (SDP) method involving all the coefficients of the dynamical system, which can be regarded as an extension of Yong \cite{Yong-2013}. Subsequently, Shi {\em et al.} \cite{Li-Shi-Yong-2021} developed the results established in Huang {\em et al.} \cite{Huang-Li-Yong-2015} to the settings of non-zero-sum games. 

In 1954, Minsky \cite{Minsky-1954} initiated the reinforcement learning (RL) concept.
 %{\color{red}As an important subfield of machine learning, RL has become one of the most popular researches in policy iterations and has been studied extensively in many different directions. 
%Unlike supervised and unsupervised learning, which require large amounts of labeled/unlabeled training data, RL samples training data from the environment. RL learns a sequence of actions in an interactive environment by trial and error that maximizes expected reward, which can be viewed as a straightforward form of adaptive optimal control (see Sutton and Barto \cite{Sutton-Barto-2018}).} 
%Compared with classical approaches to solve optimal control problems, RL approaches focus on how to learn optimal actions from past data to reinforce rewards without knowing the structure of the dynamic system. 
In recent years, many scholars have devoted themselves to RL research for deterministic optimal control problems, see 
Bradtke {\em et al.} \cite{Bradtke-Ydestie-Barto-1994}, 
Lewis {\em et al.} \cite{Lewis-Vrabie-Vamvoudakis-2012}, Chen {\em et al.} \cite{Chen-Qu-Tang-Low-Li-2022}, and so on. Despite many difficulties compared to the deterministic case, there has been growing academic interest in RL techniques for studying stochastic optimal control problems. Wang {\em et al.} \cite{Wang-Zariphopoulou-Zhou-2019} devised an {\em exploratory formulation} for a nonlinear stochastic system to capture learning under exploration, which is a revitalization of the classical stochastic relaxed control. Based on \cite{Wang-Zariphopoulou-Zhou-2019}, there are a series of follow-up works: Wang and Zhou \cite{Wang-Zhou-2019} presented the best trade-off between exploration and exploitation and devised an implementable RL algorithm to solve the mean-variance portfolio problem. Gao {\em et al.} \cite{Gao-Xu-Zhou-2022} studied the temperature control problem for Langevin diffusions non-convex optimization by the stochastic relaxed control and gave a Langevin algorithm based on the Euler-Maruyama discretization of stochastic differential equation (SDE). Different from the exploratory formulation, Li {\em et al.} \cite{Li-Li-Peng-Xu-2022} introduced a partial model-free RL method based on Bellman's dynamic programming principle for a kind of It\^o type LQ optimal control. Some other recent RL approach works in stochastic optimal control refer to Bian {\em et al.} \cite{Bian-Jiang-Jiang-2016}, Du {\em et al.} \cite{Du-Meng-Zhang-2020}, and so on.

Recently, the study of MFGs by the RL method has attracted much attention. Lauri{\`e}re {\it et al.} \cite{Lauriere et al.-2022} gave two deep RL methods for dynamic MFGs: One learns a mixed strategy by historical data, and the other is an online mixing method based on regularization without memorizing historical data or previous estimates. Perrin {\em et al.} \cite{Perrin et al.-2021} obtained a Nash equilibrium by deep RL and normalizing flows, in which the agents adapted their velocity to match the neighboring flock's average one. Elie {\em et al.} \cite{Elie et al.-2020} gave model-free learning algorithms towards non-stationary MFG equilibria relying on some classical assumptions for multi-agent systems. Xu {\em et al.} \cite{Xu-Shen-Huang-2023} presented a model-free method based on the Nash certainty equivalence-based strategy to obtain $\varepsilon$-Nash equilibria for a kind of MFGs. Their RL algorithm measures data from a selected agent as reinforcement signals. 

On the other hand, there is little literature on RL methods to study MFC problems. In reality, many phenomena in engineering control problems are involved in the MF term. For example, uncertain factors and average performances affect autonomous vehicles, energy-efficient buildings, and renewable energy, which can be modeled by the stochastic system involving the MF term. Moreover, when considering the comfort of vehicle driving, the steady temperature of the building, and the risk from uncertainties of renewable energy, one needs to minimize the variance so that the cost involves the MF term as well. In some cases, the controller may only know some information about the underlying systems but only state trajectories, which poses difficulties in finding the optimal control. Motivated by the above practical problems, we devoted ourselves in this paper to developing an RL algorithm to solve a kind of MFC-LQ problem at an infinite scale, in which we can observe data trajectories and only know partial information about the system.

Different from the discrete-time case of MFGs discussed in Lauri{\`e}re {\it et al.} \cite{Lauriere et al.-2022}, Perrin {\em et al.} \cite{Perrin et al.-2021}, and Elie {\em et al.} \cite{Elie et al.-2020}, we intend to study the continuous-time stochastic MFC-LQ problem in the infinite range by RL methods. Similar to Xu {\em et al.} \cite{Xu-Shen-Huang-2023}, we also study the policy iteration RL method based on analyzing a couple of GAREs to obtain the optimal feedback control. 
 The environment changes as time goes by, so the control system has to be modified according to the new information and initial pairs. 
Different from the MFC-LQ problem with expectation $\mathbb E$ in \cite{Huang-Li-Yong-2015}, we consider the conditional expectation $\mathbb E_t$ as the MF term in the system/cost functional. The conditional expectation $\mathbb E_t$ indicates that the MFC-LQ problem is time-inconsistent, and Bellman's dynamic programming principle is invalid. The reason is that the conditional expectation $\mathbb E_t$ makes the state not satisfy the semigroup property. It follows that the optimal control for the initial point $(t,x)$ may not minimize the cost functional for a later point on the optimal trajectory. In this case, we consider the {\it pre-commitment} optimal control studied in Yong \cite{Yong-2017}. Since the problem is considered in the infinite horizon, the stabilizability should be discussed in this paper rather than in \cite{Yong-2017}. 

Here, we devise a policy iteration method to obtain the pre-commitment optimal control, an essential iterative method in reinforcement learning, including {\em policy evaluation} and {\em policy improvement}. Policy evaluation uses the state and the control of the environment to evaluate the reward. Then, policy improvement explores a new control policy using the policy evaluation result. In this procedure, the reward is reinforced progressively, and each policy is better than the previous one. The RL algorithm ends, and the optimal control is obtained until the policy improvement step no longer changes the policy evaluation step. Policies must be stabilizable and convergent based on alternate iterations of these two steps. 

The main contributions of this paper include: 

(i) Algorithm Aspect: Different from \cite{Li-Li-Peng-Xu-2022}, conditional expectations of state and control complicate the RL method to obtain the pre-commitment optimal control. One obstacle to solving our problem is that the Bellman equation in \cite{Li-Li-Peng-Xu-2022} is invalid; another one is that the value function $V(x)=\langle\hatP x,x\rangle$ involves only $\hatP$, not $P$. To solve $(P,\hatP)$ by state trajectories rather than coupled GAREs, we creatively construct an {\em auxiliary function} and combine the {\em cost functional} as the objective functions for evaluating the reward to calculate $P^{(i+1)}$ and $\hatP^{(i+1)}$, respectively. One virtue of this algorithm is that we can calculate $P^{(i+1)}$ and $\hatP^{(i+1)}$ independently although the GAREs are coupled. Another virtue is that this method uses only partial coefficients of the system and does not need to calculate GAREs themselves, which differs from the SDP method in Huang {\em et al.} \cite{Huang-Li-Yong-2015}. 

(ii) Theoretical Aspect: Comparing to Huang {\em et al.} \cite{Huang-Li-Yong-2015} and Yong \cite{Yong-2017}, the MFC-LQ problem in this paper involves conditional expectation on $[t,\infty)$, which complicates the analysis of theoretical results. We formulate a proposition for the stabilizability of the system, which plays a key role in this paper and relaxes the equivalent condition for the stabilizer in \cite{Huang-Li-Yong-2015}. The equivalence between the RL algorithm and Lyapunov Recursion is proved to obtain the stabilizability and convergence properties for the RL algorithm, which is more complicated than \cite{Li-Li-Peng-Xu-2022} because of the invalidity of the Bellman equation.

 (iii) Numerical Implementation Aspect: The methods in Wang {\it et al.} \cite{Wang-Zariphopoulou-Zhou-2019}, Wang and Zhou \cite{Wang-Zhou-2019}, and Gao {\it et al.} \cite{Gao-Xu-Zhou-2022} intend to obtain the optimal control distribution by {\em exploratory formulation} and only deal with one-dimensional numerical example. Differently, our RL algorithm obtains the exact optimal control by {\em policy iteration} and deals with the numerical example in high dimension by the Kronecker product. The implementation in \cite{Li-Li-Peng-Xu-2022} only calculated $P^{(i+1)}$ using Bellman equation on $[t, t+\Delta t]$, however, the calculation for $(P^{(i+1)},\hatP^{(i+1)})$ in this paper is based on two equations on $[t, \infty)$: a new equation involving integrals on both sides is introduced to solve $P^{(i+1)}$ and an equation involving value function and cost functional is used to solve $\hatP^{(i+1)}$.

The RL method for MFC-LQ problems introduced in this paper is valuable and useful for engineering applications. Let us return to the phenomena mentioned in the previous motivation. As an operator, one may not need to learn the system's internal structure. During the operation process, one can only observe the state of the autonomous vehicles, the temperature of the building, and the capacity of wind or solar power. For example, a controller may want to increase the capacity of a wind power generation system; meanwhile, he may hope to minimize the variance to resist risk from the uncertainty of the wind. With partial information of the system, he tries, based on his experience, to run the system. By observing the state trajectories, he evaluates the reward (policy evaluation) and then improves the control (policy improvement). By repeating this procedure, he can obtain the optimal control using the RL method.

The rest of this paper is organized as follows. Section \ref{Sec2} introduces an MFC-LQ problem and some preliminaries. In Section \ref{Sec3}, we present a policy iteration RL algorithm to compute the pre-commitment optimal feedback control and discuss the stabilizability and convergence of the algorithm. We give an algorithm implementation and illustrate it with a numerical example in Section \ref{Sec4}.

\emph{Notation:} Let $(\Omega, \cF, \bbP,\bbF)$ be a complete filtered probability space on
which a standard one-dimensional Brownian motion $W(\cdot)$ is
defined with $\bbF\equiv\{\cF_t\}_{t\geq 0}$ being its
natural filtration augmented by all $\bbP$-null sets. Let $\bbN$ denote the set of positive integers, and $l, m, n, k, L, M, N, H\in
\bbN$ be the given constants. Denote $ \bbR^n$ as the $n$-dimensional
Euclidean space with the norm $|\cdot|$. Hilbert space 
$L^2_{\bbF}([t,\infty); \bbR^n)$ is defined as the space of $ \bbR^n$-valued $\bbF$-progressively measurable processes $\varphi(\cdot)$ with the finite norm 
$$\|\varphi(\cdot)\|=\left[\E_t\int_{t}^\infty |\varphi(s)|^2\ds\right]^{\frac{1}{2}}<\infty.$$
Here, ${\E_t}={\E}[\cdot|\cF_t]$ stands for the conditional expectation operator. 
Let $ \bbR^{n\times m}$ be the set of all
$n\times m$ real matrices and $\cS^n$ be the collection of all symmetric matrices in $ \bbR^{n\times n}$. As usual, if a
matrix $A\in \cS^n$ is positive semidefinite (resp. positive
definite), we write
$A\geq 0$ (resp. $>0$). All the positive
semidefinite (resp. positive
definite) matrices are collected by
$\cS^n_+$ (resp. $\cS^n_{++}$). 
If $A$, $B\in \cS^n$, then we write $A\geq B$ (resp. $>$) if $A-B\geq 0$ (resp. $>0$). 
%Denote $s, t\geq0$ as the time in infinite horizon. 
Furthermore, $O$ denotes zero matrices with appropriate dimension, $I$ denotes identity matrices with appropriate dimensions, and the superscript $\top$ denotes the transpose of a matrix (or a vector).

For any $0\leq t<T<\infty$, we introduce the following spaces:
\begin{itemize} 
\item $\cX[t,T]=\Big\{X : [t,\infty)\times \Omega \rightarrow \bbR^n~\Big|~X(\cdot)$ is $\bbF$-adapted, $t\rightarrow X(t,\omega)$ is continuous, and $\E_t\Big[\max_{s\in[t,T]}|X(s)|^2\Big] <\infty\Big\} $;
\item $\cX_{\loc}[t,\infty)=\bigcap_{T>t}\cX[t,T]$;
\item $\cX[t,\infty)=\Big\{X(\cdot)\in\cX_{\loc}[t,\infty)~\Big|~\E_t\Big[\int_t^\infty|X(s)|^2\ds\Big]<\infty\Big\}$.
\end{itemize}

\section{Problem Formulation and Preliminaries }\label{Sec2}
\subsection{Problem Formulation}
In this paper, we consider the following controlled
MF-SDE on $[t,\infty)$ for $t\geq 0$:
\begin{equation}\label{MF-system}
\left\{
\begin{aligned}
& \dd X(s)=\Big\{AX(s)+\barA{\E_t} [X(s)]+Bu(s)+\barB{\E_t} [u(s)] \Big\}\ds \\
& \quad+\Big\{ CX(s)+\barC{\E_t} [X(s)]+Du(s)+\barD{\E_t} [u(s)] \Big\} \dw(s),\\&\qquad \qquad \qquad \qquad \qquad \qquad   \qquad \qquad \qquad \qquad s\in [t,\infty),\\
& X(t)=x,
\end{aligned}
\right.
\end{equation}
where $A$, $\barA$, $C$, $\barC\in \bbR^{n\times n}$ and $B$, $\barB$, $D$, $\barD\in \bbR^{n\times m}$ are constant matrices of proper sizes. 
The process $X(\cdot)$ valued in $ \bbR^n$ is called a {\it state process} corresponding to the control $u(\cdot)$, where $u(\cdot) \in L^2_{\bbF}([t,\infty); \bbR^m)$ is the {\it control process}; and $(t,x)\in\cD$ is called the initial pair, where 
\[
\cD=\big\{(t,x)~\big|~t\in[0,\infty), x {\rm~is~} \cF_t{-\rm measurable}, \E(|x|^2)<\infty\big\}.
\] 
%Note that for any $(t,x)\in\cD$, the corresponding state process $X(\cdot) \equiv X(\cdot;t,x)$ depends on $(t,x)$. 
For any initial pair $(t,x)\in\cD$ and   control $u(\cdot)\in L^2_{\bbF}([t,\infty); \bbR^m)$, MF-SDE \eqref{MF-system} admits a unique
solution $X(\cdot) \equiv X(\cdot;t,x,u(\cdot)) \in \cX_{\loc}[t,\infty)$ (see, Huang {\it et al.} \cite{Huang-Li-Yong-2015}). Our cost functional is defined as follows:
\begin{equation}\label{Sec2_Cost}
\begin{aligned}
&~~~J\big(t,x;u(\cdot)\big)\\=\ & {\E_t} \bigg\{ \int_t^\infty \Big[
\big\langle QX(s),\ X(s) \big\rangle+\big\langle \barQ{\E_t} [X(s)],\ {\E_t} [X(s)] \big\rangle\\
&\qquad \quad+2\big\langle SX(s),\ u(s) \big\rangle+2\big\langle \barS{\E_t} [X(s)],\ {\E_t} [u(s)] \big\rangle\\
&\qquad \quad+\big\langle Ru(s),\ u(s) \big\rangle+\big\langle \barR{\E_t} [u(s)],\ {\E_t} [u(s)] \big\rangle \Big]\ds\bigg\},
\end{aligned}
\end{equation}
where $Q$, $\barQ$, $S$, $\barS$, $R$ and $\barR$ are given constant matrices of proper sizes. Note that in general, given an initial $(t,x)\in\cD$ and a control $u(\cdot)\in L^2_{\bbF}([t,\infty); \bbR^m)$, the solution $X(\cdot)$ of \eqref{MF-system} might just be in $\cX_{\loc}[t,\infty)$, which is not enough to ensure the cost functional $J\big(t,x;u(\cdot)\big)$  well-defined. 
To avoid such situations, we introduce the following set
\begin{align*}
	&\cU_{ad}[t,\infty)
=\Big\{u(\cdot)\in L^2_{\bbF}([t,\infty); \bbR^m)~\big|\\\nonumber & ~~~~~~~~~~~~X(\cdot)\equiv X(\cdot;t,x,u(\cdot))\in\cX[t,\infty)  {\rm~for ~any~}(t,x)\in\cD \Big\}.\nonumber
\end{align*}
The elements of $\cU_{ad}[t,\infty)$ are called {\it admissible control processes} and the corresponding states are called {\it admissible state processes}. 

We introduce a family of MFC-LQ stochastic optimal control problems.
 \smallskip\\
\noindent{\bf Problem (MFC-LQ).} For any $(t,x)\in\cD$, find an admissible
control $u^*(\cdot) \in \cU_{ad}[t,\infty)$ such that
\begin{equation*}%\label{Sec2_Essinf}
J\big(t,x;u^*(\cdot)\big)=V(t, x)\triangleq\inf_{u(\cdot)\in \cU_{ad}[t,\infty)} J\big(t,x;u(\cdot) \big).
\end{equation*}
Here, $V(t,x)$ is called the {\it pre-commitment} value function. 
For any given $(t,x)\in\cD $, 
Problem (MFC-LQ) is called {\it well-posed} at $(t,x)$ if $ V(t, x)$ is finite. 
It is called {\it solvable} at $(t,x)$ if Problem (MFC-LQ) is well-posed and $ V(t, x)$ is achieved by an
admissible control $u^*(\cdot)\in\cU_{ad}[t,\infty)$. 
Correspondingly, $u^*(\cdot)$ is called a {\it pre-commitment optimal control}, 
$X^*(\cdot) \equiv X\big(\cdot;t, x, u^*(\cdot)\big)$ called the
{\it pre-commitment optimal trajectory}, and
$(X^*(\cdot), u^*(\cdot))$ called a {\it pre-commitment optimal pair}.

For simplicity, we adopt the notations $\widehat \Pi=\Pi+\overline\Pi$ with $\widehat \Pi=\hatA, \hatB, \hatC, \hatD, \hatQ, \hatS, \hatR$, $\Pi=A, B, C, D, Q, S, R$, and $\overline\Pi=\barA, \barB, \barC, \barD, \barQ, \barS, \barR$. Also, we denote $\bf\Gamma=\left(\begin{array}{ccc} \Gamma & O\\ O & \widehat
\Gamma
\end{array} \right)$ with $\bf\Gamma=\bfQ, \bfS, \bfR$, $\Gamma=Q, S, R$, and $\widehat \Gamma=\hatQ, \hatS, \hatR$ in this paper.

Similar to the condition {\bf (J)$'$} in Huang {\em et al.} \cite{Huang-Li-Yong-2015}, we put the following positive definite condition (PDC, for short) on the triple $(\bfQ, \bfS, \bfR)$. 
\begin{equation}\label{CPD}
 \mbox{\bf Condition (PDC).}\quad\mbox{$\bfR > 0$ and $\bfQ-\bfS^{\top} \bfR^{-1}\bfS>0$. \hspace{181pt}} 
\end{equation}
It is obvious that, if $(\bfQ, \bfS, \bfR)$ satisfies the PDC \eqref{CPD}, then we have
$J(t,x;u(\cdot)) > 0$ for any $(t,x)\in \cD$ and
any $u(\cdot)\in\cU_{ad}[t,\infty)$ except for the equilibrium $x=0$ so that Problem (MFC-LQ) is
well-posed. As stated in Huang {\em et al.} \cite{Huang-Li-Yong-2015}, if the PDC \eqref{CPD} holds, $u(\cdot) \in \mathcal{U}_{a d}[t, \infty)$ if and only if for any $(t,x) \in \mathcal D$, $X(\cdot ; x, u(\cdot)) \in \mathcal{X}[t, \infty)$.

% {\color{red}It worth pointing out that we do not consider the boundary case: $\bfR > 0$ and $\bfQ-\bfS^{\top} \bfR^{-1}\bfS=0$. In this case, one can check by completing square that $$u^*=-R^{-1}SX+(R^{-1}S-\hatR^{-1}\hatS)\E_t[X]$$ is the pre-commitment optimal feedback control that makes $J(t,x;u)$ to achieve its minimum with $0$. }
\subsection{MF-$L^2$-Stabilizable}

\begin{definition}\label{stabilizition}
System \eqref{MF-system} is said to be MF-$L^2$-stabilizable if there exists a pair $(K,\hatK)\in(\bbR^{m\times n})^2$ such that for any $(t,x)\in\cD $ the solution of 
{\small\begin{equation}\label{MF-systemK}
\left\{
\begin{aligned}
& \dd X(s)=\Big\{ \big(A+BK\big)\big(X(s)-{\E_t}[X(s)]\big)+\big(\hatA+\hatB\hatK\big){\E_t}[X(s)]\Big\}\ds \\
& ~~+\Big\{ \big(C+DK\big)\big(X(s)-{\E_t}[X(s)]\big)+\big(\hatC+\hatD\hatK\big){\E_t}[X(s)]\Big\} \dw(s),\\&~\qquad\qquad\qquad\qquad\qquad\qquad\qquad\qquad\qquad\qquad\qquad\quad s\in [t,\infty),\\
& X(t)=x
\end{aligned}
\right.
\end{equation}}
satisfies $\lim\limits_{s\rightarrow\infty}{\E} |X(s)|^2=0$ and $X(s)\in\cX[t,\infty)$. In this case, $(K,\hatK)$ is called an MF-$L^2$-stabilizer of system \eqref{MF-system}. 
\end{definition}

\begin{assumption}\label{ass1}
System \eqref{MF-system} is MF-$L^2$-stabilizable.
\end{assumption}

Next, we present an equivalent condition for the existence of the stabilizers for the system \eqref{MF-system}, which in particular shows that the stabilizability  of \eqref{MF-system} is independent of the initial point $(t,x)\in\cD$. 
\begin{proposition}
\label{lemma-stabilizer} 
A pair $(K,\hatK)\in(\bbR^{m\times n})^2$ is an MF-$L^2$-stabilizer of the system \eqref{MF-system} if and only if 
there exists a pair of matrices $(P,\hatP)\in(\cS_{++}^n)^2$ such that 
\begin{equation}\label{stabilizable-condition}
\left\{
\begin{aligned}
&(A+BK)^{\top} P+P(A+BK)
+(C+DK)^{\top} P(C+DK)<0,\\
&(\hatA+\hatB\hatK)^{\top}\hatP+\hatP(\hatA+\hatB\hatK)<0.
\end{aligned}
\right.
\end{equation}
In this case, for any $\Lambda,\widehat \Lambda\in\cS^n$ (resp., $\cS_{+}^n$, $\cS_{++}^n$), the system of Lyapunov equations
{\small\begin{equation}\label{stabilizition-2}
\left\{
\begin{aligned}
&(A+BK)^{\top} P+P(A+BK)
+(C+DK)^{\top} P(C+DK)+\Lambda=0,\\
&(\hatA+\hatB\hatK)^{\top}\hatP+\hatP(\hatA+\hatB\hatK)+\widehat \Lambda=0
\end{aligned}
\right.
\end{equation}}
admits a unique solution $(P,\hatP)\in(\cS^n)^2$ (resp., $(\cS_{+}^n)^2$, $(\cS_{++}^n)^2$). 
\end{proposition}
\begin{proof}
``$\Leftarrow$": We first prove that $(K,\hatK)\in(\bbR^{m\times n})^2$ is an MF-$L^2$-stabilizer of the system \eqref{MF-system} if \eqref{stabilizable-condition} holds. 

% We claim that if the controlled ODE system $[\hatA;\hatB]$ is stabilizable, and controlled SDE system $[A,C;B,D]$ is $L^2$-stabilizable, then the controlled MF-FSDE system $[A,\barA, C, \barC; B, \barB, D, \barD]$ is MF-$L^2$-stabilizable. 
Taking conditional expectation on \eqref{MF-systemK}, we have
\begin{equation}\label{MF-Ex}
\left\{
\begin{aligned}
& d {\E_t} [X]=(\hatA+\hatB\hatK){\E_t} [X]\ds,\quad s\in [t,\infty),\\
& {\E_t}[X(t)]=x.
\end{aligned}
\right.
\end{equation}
For any $P\in\cS^n$, applying It\^o's formula to $\langle P(X-{\E_t}[X]), X-{\E_t}[X]\rangle$ with $X-{\E_t}[X]$ calculated by \eqref{MF-systemK} and \eqref{MF-Ex}, then integrating on $[t,\tau]$ for any constant $\tau>t$ and taking conditional expectation ${\E_t}[\cdot]$ on both sides yields 
\begin{equation}\label{var}
\begin{aligned}
&~~~{\E_t}\big\langle P(X(\tau)-{\E_t}[X(\tau)]), X(\tau)-{\E_t}[X(\tau)] \big\rangle\\&={\E_t}\int_{t}^\tau\Big\{\Big\langle\Big((A+BK)^{\top} P+P(A+BK)\\&
{\small+(C+DK)^{\top} P(C+DK)\Big)\big(X(s)-{\E_t}[X(s)]\big), X(s)-{\E_t}[X(s)]\Big\rangle} \\&~~~+\Big\langle
(\hatC+\hatD\hatK)^{\top} P(\hatC+\hatD\hatK){\E_t}[X(s)], {\E_t}[X(s)]\Big\rangle\Big\}\ds.
\end{aligned}
\end{equation} 
If there exists a $P\in\cS_{++}^n$ such that 
the first inequality of \eqref{stabilizable-condition} holds,
%, for $\Lambda=I$, we have 
then all the eigenvalues of 
\[
(A+BK)^{\top} P+P(A+BK)
+(C+DK)^{\top} P(C+DK)
\]
are negative. 
Denote 
$\varphi_t(\tau)=\var_t[X(\tau)]$ and $\psi_t(\tau)=|{\E_t}[X(\tau)]|^2$, 
then \eqref{var} implies that 
$
\begin{aligned}
\dot \varphi_t(\tau)\leq -\mu\varphi_t(\tau)+L\psi_t(\tau),
\end{aligned}
$
for some $\mu>0$ and $L>0$. 
By the Gronwall inequality in differential form and using $\varphi_t(t)=0$, we have $\varphi_t(\tau)\leq L\int_t^\tau e^{-\mu(\tau-s)}\psi_t(s)\ds$. 
Solving \eqref{MF-Ex}, we obtain $
{\E_t} [X(\tau)]=e^{(\hatA+\hatB\hatK)(\tau-t)}x. 
$

Because the second inequality of \eqref{stabilizable-condition} is the sufficient and necessary condition of stabilizable property for the system \eqref{MF-Ex}, there exist $M>0$ and $m>0$ with $\mu\neq m=-2\max\sigma(\hatA+\hatB\hatK)$ such that $\psi_t(\tau)\leq M|x|^2e^{-m (\tau-t)},~\tau\geq t$. Then 
\[
\begin{aligned}
\varphi_t(\tau)&\leq LM|x|^2\frac{e^{-(\mu+m) \tau }-e^{-\mu(\tau-t) }}{\mu-m}, \end{aligned}
\]
which implies that
\[
\begin{aligned}
&~~~{\E}|X(\tau)|^2={\E}\big(\varphi_t(\tau)+\psi_t(\tau)\big)\\
&\leq M{\E}|x|^2\Big[L\frac{e^{-(\mu+m) \tau }-e^{-\mu(\tau-t)}}{\mu-m}+e^{-m (\tau-t)}\Big]\rightarrow 0,~\tau \rightarrow \infty.
\end{aligned}
\]
Because the inequalities of \eqref{stabilizable-condition} hold, from Proposition 2.2 in \cite{Li-Shi-Yong-2021}, $X(s)\in\cX[t,\infty)$, thus $(K,\hatK)$ is an MF-$L^2$-stabilizer for the system \eqref{MF-system}. 

``$\Rightarrow$" 
Since $\hatP\in\cS_{++}^n$, then $(\hatC+\hatD\hatK)^{\top}\hatP(\hatC+\hatD\hatK)\geq0$. From Proposition A.5 in \cite{Huang-Li-Yong-2015}, it is obvious that \eqref{stabilizable-condition} holds. By Theorem 1 in \cite{Rami-Zhou-2000}, the Lyapunov equations \eqref{stabilizition-2} admit a unique solution $(P,\hatP)\in(\cS^n)^2$ (resp., $(\cS_{+}^n)^2$, $(\cS_{++}^n)^2$) for any $\Lambda,\widehat \Lambda\in\cS^n$ (resp., $\cS_{+}^n$, $\cS_{++}^n$). 
\end{proof}
 
Theoretically, Proposition \ref{lemma-stabilizer} gives an equivalent condition to check the stabilizer. In practice, when only partial coefficients of the system (1) are known, one can try some $(K,\widehat K)$ to run the system and observe the state trajectory $X(s)$. If the chosen $(K,\widehat K)$ can make  $X(s)$ tend to a neighborhood of zero as time $s$ goes to infinity, then $(K,\widehat K)$ can be chosen as a stabilizer and Assumption 2.1 is satisfied. 

\section{A Policy Iteration Reinforcement Learning Algorithm for MFC-LQ Problem}\label{Sec3}

In this section, we present a policy iteration RL algorithm to solve the pre-commitment optimal control of Problem (MFC-LQ). We begin by resolving the solvability issue of the corresponding GAREs
\begin{equation}\label{Riccati_Eq1}
\left\{
\begin{aligned}
& A^{\top} P+PA+C^{\top} PC+Q -\big(PB+C^{\top}
PD+S^{\top}\big)\\
&~~~~~~~~~\times \big(D^{\top} PD+R\big)^{-1}
\big(B^{\top} P+D^{\top}
PC+S\big)=0,\\
& D^{\top} PD+R > 0,\end{aligned}
\right.
\end{equation}
\begin{equation}\label{Riccati_Eq2}
\left\{
\begin{aligned}
&\hatA^{\top}\hatP+\hatP\hatA+\hatC^{\top} P
\hatC+\hatQ -\big(\hatP\hatB 
+\hatC^{\top} P\hatD+\hatS^{\top}\big)\\
&~~~~~~~~~\times\big(\hatD^{\top} P
\hatD+\hatR \big)^{-1} \big(\hatB^{\top}\hatP 
+\hatD^{\top} P\hatC+\hatS\big)
=0,\\
&\hatD^{\top} P\hatD+\hatR > 0.\end{aligned}
\right.
\end{equation}
%{\color{blue} The classical Riccati equations, arising from deterministic LQ problems (see \cite{Bittanti-Laub-Willems-1991}), have the quadratic form. 
% When faced with stochastic LQ problems, one has 
%{\em generalized} or {\em formal} Riccati equations (see \cite{Bismut-1976}). Although they are not of quadratic form, they are still named after Riccati in the literature, such as \cite{Rami-Zhou-2000} and \cite{Huang-Li-Yong-2015}. In the cases of $D=O$ or $m=1$, \eqref{Riccati_Eq1} will be reduced to the classical quadratic form. As pointed out in \cite{Rami-Zhou-2000}, GARE \eqref{Riccati_Eq1} are fundamentally different from their deterministic counterpart algebraic Riccati equations and more complicated since the inverse part involves the unknown $P$. Unlike \eqref{Riccati_Eq1}, \eqref{Riccati_Eq2} is of the classical quadratic form since the inverse part only involves the known $P$ but not the unknown $\widehat P$.} 
The classical Riccati equations, which emerge from deterministic LQ problems (see \cite{Bittanti-Laub-Willems-1991}), have a quadratic form. In contrast, when dealing with stochastic LQ problems, one encounters {\em generalized} or {\em formal} Riccati equations (see \cite{Bismut-1976}). Although these do not have a quadratic form, they are still referred to as Riccati equations in the literature (see \cite{Rami-Zhou-2000} and \cite{Huang-Li-Yong-2015}). In the cases of $D=O$ or $m=1$, \eqref{Riccati_Eq1} is  simplified to the classical quadratic form. As noted in \cite{Rami-Zhou-2000}, the  GARE \eqref{Riccati_Eq1} is fundamentally different from their deterministic counterparts and is more complex, as the inverse component involves the unknown $P$. Unlike \eqref{Riccati_Eq1}, \eqref{Riccati_Eq2} retains the classical quadratic form because the inverse component involves only the known  $P$, not the unknown $\widehat P$.

Now, we give the unique solvability of the GAREs \eqref{Riccati_Eq1}-\eqref{Riccati_Eq2} and construct a pre-commitment optimal control of feedback form as follows.

\begin{theorem}\label{th-1} \sl
Under Assumption \ref{ass1} and the PDC \eqref{CPD}, the
 system of the GAREs \eqref{Riccati_Eq1}-\eqref{Riccati_Eq2} admits a unique pair of solution $(P,\hatP)\in(\cS^n_{++})^2$. Moreover, for any given
$(t,x)\in\cD$, the following feedback form control 
\begin{equation}\label{Optimal_Control}
u^*(s)=KX^*(s)+(\hatK-K)\mathbb E_t[X^*(s)],\quad s\in [t,\infty)
\end{equation}
is the unique optimal control of Problem (MFC-LQ), 
where $(K,\hatK)$ is a stabilizer given by 
\begin{equation}\label{KK}
\left\{\begin{aligned}
& K=- \big(D^{\top} PD+R\big)^{-1}
\big(B^{\top} P+D^{\top}
PC+S\big),\\
&\hatK=- \big(\hatD^{\top} P
\hatD+\hatR \big)^{-1} \big(\hatB^{\top}\hatP 
+\hatD^{\top} P\hatC+\hatS\big) 
\end{aligned}\right.
\end{equation}
and the optimal trajectory $X^*$ is determined by
\begin{equation*}\label{Sec2.F_Optimal_State}
\left\{
\begin{aligned}
& \dd X^*=\Big\{ \big(A+BK \big) \big(X^*-{\E_t} [X^*]\big)+\big(\hatA+\hatB\hatK \big){\E_t}[X^*]\Big\}\ds\\
& +\Big\{ \big(C+DK \big) \big(X^* -{\E_t}[X^*] \big)+\big(\hatC+\hatD\hatK\big){\E_t}[X^*] \Big\}\dw(s), \\
&X^*(t)=x, \quad \qquad \qquad \qquad \qquad \qquad \qquad \qquad \qquad s\in [t,\infty).
%& X^*(t)=x.
\end{aligned}
\right.
\end{equation*}
Moreover, 
\begin{equation}\label{V(x)}
	V(x)=\big\langle\hatP x,x\big\rangle
\end{equation}
is the  value function and  $V(x)>0$ except for $x=0$.
\smallskip
\end{theorem}
\begin{proof}
We firstly prove that $(X^*,u^*)$ is the unique pre-commitment optimal pair of Problem (MFC-LQ). By Theorem 5.2 in \cite{Huang-Li-Yong-2015}, 
let $
(P,\widehat P)$ be the solution of the GAREs \eqref{Riccati_Eq1}-\eqref{Riccati_Eq2}. Based on \eqref{MF-systemK} and \eqref{MF-Ex}, applying It\^o's formula to $\langle P(X-{\E_t} [X]),~X-{\E_t} [X]\rangle$ and differentiating $\langle\hatP{\E_t} [X],~{\E_t} [X]\rangle$, by completing squares, we have 
{\small\begin{equation*}\label{J1squared}
\begin{aligned}
&~~~J(t,x;u(\cdot))\\&={\E_t}\int_t^\infty\Big\{\big\langle (D^{\top} PD+R)\\&\qquad\times\big[u-{\E_t} [u]-K(X-{\E_t} [X])\big],~\big[u-{\E_t} [u]-K(X-{\E_t} [X])\big]\big\rangle\\
&~~~\qquad\quad+\big\langle\big[ A^{\top} P+PA+C^{\top} PC+Q-K^{\top}(D^{\top} PD+R)K\big]\\&\qquad\qquad\times(X-{\E_t} [X]),~X-{\E_t} [X]\big\rangle\Big\}\ds\\
&~~~+\int_t^\infty\Big\{\big\langle (\hatD^{\top} P\hatD+\hatR)({\E_t} [u]-\hatK{\E_t} [X]),~{\E_t} [u]-\hatK{\E_t} [X]\big\rangle\\
&~~~\qquad+\big\langle\big[\hatA^{\top}\hatP+\hatP\hatA+\hatC^{\top} P
\hatC+\hatQ\\
&~~~\qquad -\hatK^{\top} \big(\hatD^{\top} P\hatD+\hatR \big)\hatK\big]{\E_t} [X],~{\E_t} [X]\big\rangle\Big\}\ds+\big\langle\hatP x,~x\big\rangle.
\end{aligned}
\end{equation*}}%%
Recalling that $(P,\hatP)$ is the solution of  \eqref{Riccati_Eq1}-\eqref{Riccati_Eq2}, $D^{\top} PD+R>0$ and $\hatD^{\top} P\hatD+\hatR>0$, we get $J(t,x;u(\cdot))\geq\big\langle\hatP x,~x\big\rangle$.
If we take $u^*(\cdot)$ in the form of \eqref{Optimal_Control}, then \eqref{V(x)} holds true and \eqref{Optimal_Control} is the optimal feedback control. Under the PDC \eqref{CPD}, similar discussion to Theorem 5.2 in \cite{Huang-Li-Yong-2015}, the optimal control $u^*(\cdot)$ and optimal state $X^*(\cdot)$ are unique. From \eqref{Optimal_Control}, we have ${\E_t} [u^*(s)]=\hatK{\E_t} [X^*(s)]=\hatK e^{(\hatA+\hatB\hatK)(s-t)}x$ is also unique. Since $e^{(\hatA+\hatB\hatK)(s-t)}$ is invertible and $x$ is arbitrary, then $\hatK$ is unique. Similarly, $K$ is also uniquely determined. 

By \eqref{KK}, GARE \eqref{Riccati_Eq1} can be rewritten as
\begin{equation}\label{Lyapunov K}
\begin{aligned}
&(A+BK)^{\top} P+P(A+BK)+(C+DK)^{\top} P(C+DK)\\
&\qquad\qquad\qquad+K^{\top} RK+S^{\top} K+K^{\top} S+Q=0. 
\end{aligned}
\end{equation} 
Similar to Lemma 2.1 in \cite{Sun-Yong-2018}, $P$ can be uniquely solved as
\begin{equation}\label{P}
	P=\mathbb E_t\int_t^\infty\Phi(s)^\top(K^{\top} RK+S^{\top} K+K^{\top} S+Q)\Phi(s)\ds,
	\end{equation}
where $\Phi(\cdot)$ is the solution to the following SDE for $\mathbb R^{n\times n}$-valued processes:
\begin{equation*}
\left\{
\begin{aligned}
		\dd\Phi(s)&=(A+BK)\Phi(s)\ds+(C+DK)\Phi(s)\dw(s),~s\in[t,\infty),\\
		\Phi(t)&=I.
	\end{aligned}
\right.
\end{equation*} 
Because the triple $(\bfQ, \bfS, \bfR)$ satisfies the PDC \eqref{CPD}, we have 
\begin{equation}\label{Q}
\begin{aligned}
&~~~K^{\top} RK+S^{\top} K+K^{\top} S+Q>0.
\end{aligned}
\end{equation} 
Since $\Phi(\cdot)$ is invertible (see \cite{Yong-Zhou-1999}), by \eqref{P}, we have $P\in\mathcal S^n_{++}$. Combining \eqref{Lyapunov K} and \eqref{Q}, we obtain 
\begin{equation}\label{Lyapunov KK}
	(A+BK)^{\top} P+P(A+BK)
+(C+DK)^{\top} P(C+DK)<0.
\end{equation} 
By Proposition \ref{lemma-stabilizer}, we see that $K$ is the component of stabilizer.

Similarly, we also have $\hatP\in\mathcal S^n_{++}$ and $(K,\hatK)$ is a stabilizer of system \eqref{MF-system}. Since $\widehat P$ is positive definite, we see from the equation \eqref{V(x)} that $V(x)>0$ except for $x=0$. The proof is completed.

\end{proof}

From Theorem \ref{th-1}, the pre-commitment optimal control law is $u^*(X,\bar X)=KX+(\hatK-K)\bar X$, which is independent of time. So, the time-inconsistent property for this MFC-LQ problem is due to the property of state $X(\cdot)$. If the environment changes as time goes by, the controller may restart his program with a new initial state. Although the previous optimal control is not applicable anymore, the controller can still use the past optimal control gain $(K,\hatK)$ at  the new stage. It allows us to take a policy iteration RL algorithm to solve the problem at different stages. 

\begin{algorithm}\caption{\small {\it Policy Iteration for Problem (MFC-LQ)}}
1: {\bf Initialization:} Select any stabilizer $(K^{(0)},\hatK^{(0)})$ for the system \eqref{MF-system}.\\
2: Let $i=0$ and $\varepsilon>0$.\\
3: {\bf do} \{ \\
4: Obtain the state trajectory $X^{(i)}$ by running the system \eqref{MF-systemK} with $(K^{(i)},\hatK^{(i)})$ on $[t,\infty)$.\\
5: {\bf Policy Evaluation} (Reinforcement): Solve $(P^{(i+1)},\hatP^{(i+1)})$ from the identities 
{\small\begin{equation}
\begin{aligned}
&~~~\int_t^{\infty}\Big\langle P^{(i+1)}(\hatC+\hatD\hatK^{(i)}){\E_t}[X^{(i)}(s)],(\hatC+\hatD\hatK^{(i)}){\E_t}[X^{(i)}(s)]\Big\rangle\ds\\&={\E_t}\int_t^{\infty}\Big\langle(Q+2S^{\top} K^{(i)}+K^{(i)\top} RK^{(i)})
\\& \qquad \qquad \times(X^{(i)}(s)-{\E_t}[X^{(i)}(s)]),X^{(i)}(s)-{\E_t}[X^{(i)}(s)]\Big\rangle\ds,
\label{algorithm 1-evaluation-1}
\end{aligned}
\end{equation}}
{\small\begin{equation}
\begin{aligned} 
&~~~\big\langle\hatP^{(i+1)}x,x\big\rangle \\&={\E_t}\int_t^{\infty}\Big\langle(Q+2S^{\top} K^{(i)}+K^{(i)\top} RK^{(i)})
\\& \qquad \qquad \times(X^{(i)}(s)-{\E_t}[X^{(i)}(s)]),X^{(i)}(s)-{\E_t}[X^{(i)}(s)]\Big\rangle\ds\\
&+\int_t^{\infty}\Big\langle\big(\hatQ+2\hatS^{\top}\hatK^{(i)}+\hatK^{(i)\top}\hatR\hatK^{(i)}\big){\E_t}[X^{(i)}(s)],{\E_t}[X^{(i)}(s)]\Big\rangle\ds.\label{algorithm 1-evaluation-2}
\end{aligned}
\end{equation}}\\
6: {\bf Policy Improvement} (Update): Update $K^{(i+1)}$ and $\hatK^{(i+1)}$ by 
{\small\begin{equation}\label{algorithm 1-improvement}
\begin{aligned}
K^{(i+1)}&=-(R+D^{\top} P^{(i+1)} D)^{-1}(B^{\top} P^{(i+1)}+D^{\top} P^{(i+1)}C+S),
\end{aligned}
\end{equation}}
{\small\begin{equation}\label{algorithm 1-improvement-2}
\begin{aligned}
\hatK^{(i+1)}&=-(\hatR+\hatD^{\top} P^{(i+1)}\hatD)^{-1}(\hatB^{\top}\hatP^{(i+1)}+\hatD^{\top} P^{(i+1)}\hatC+\hatS).
\end{aligned} 
\end{equation} }
\\
7: If $\|P^{(i+1)}-P^{(i)}\|<\varepsilon$ and $\|\hatP^{(i+1)}-\hatP^{(i)}\|<\varepsilon$, then {\bf stop}.\\
8: $i\leftarrow i+1$ and go to step 3. \}\smallskip\\
\label{algorithm}
\end{algorithm}

Huang {\em et al.} \cite{Huang-Li-Yong-2015} solved the GAREs \eqref{Riccati_Eq1}-\eqref{Riccati_Eq2} to get $(P,\hatP)$ using SDP method. Their method necessitates all of the coefficient information in the system and is thus offline. Instead of solving Ricctai equations directly, we propose calculating $(P,\hatP)$ using Algorithm \ref{algorithm}. This method does not require all system coefficient information and observes the trajectories online. 

Denote the right sides of \eqref{algorithm 1-evaluation-1} and \eqref{algorithm 1-evaluation-2} as the objective functions $\cJ_0(t,x;K^{(i)},X^{(i)})$ and $\cJ(t,x;K^{(i)},X^{(i)})$. Our method indeed focuses on reinforcing the objective functions to compute the pair $(P^{(i+1)},\hatP^{(i+1)})$ and the control gain $(K^{(i+1)},\hatK^{(i+1)})$, respectively.
Algorithm \ref{algorithm} does not involve the coefficients $A$ and $\hatA$, so it is a partially model-free algorithm. In fact, the system's information is already embedded in the state trajectory. The other coefficients $C, \barC, B, \barB, D$, and $\barD$ in the system \eqref{MF-system} are used to improve the policy in \eqref{algorithm 1-improvement}-\eqref{algorithm 1-improvement-2}.

To guarantee the Algorithm \ref{algorithm} work, at each step $i$, we need to prove that the updated control gains $K^{(i)}$ and $\hatK^{(i)}$ in the Policy Improvement are stabilizers and the pair $(P^{(i)},\hatP^{(i)})$ is unique solvable. Moreover, we also need to make sure the sequence $\{(P^{(i)},\hatP^{(i)})\}_{i=0}^{\infty}$ is convergent. 
We will not directly establish these properties for the Algorithm \ref{algorithm}; instead, we devise a new algorithm in the following part and then show that our Algorithm \ref{algorithm} inherits these properties from this new algorithm. 

We first define {\bf Lyapunov Recursion} as follows: 
\begin{equation}\label{Lyapunov Recursion}
\begin{aligned}
&(A+BK^{(i)})^{\top} P^{(i+1)}+P^{(i+1)}(A+BK^{(i)})\\
&~~~+(C+DK^{(i)})^{\top} P^{(i+1)}(C+DK^{(i)})\\
&~~~\qquad+K^{(i)\top} RK^{(i)}+S^{\top} K^{(i)}+K^{(i)\top} S+Q=0
\end{aligned}
\end{equation}
and 
\begin{equation}\label{Lyapunov Recursion-2}
\begin{aligned}
&(\hatA+\hatB\hatK^{(i)})^{\top}\hatP^{(i+1)}+\hatP^{(i+1)}(\hatA+\hatB\hatK^{(i)})\\
&~~~+(\hatC+\hatD\hatK^{(i)})^{\top} P^{(i+1)}(\hatC+\hatD\hatK^{(i)})\\
&~~~\qquad+\hatK^{(i)\top}\hatR\hatK^{(i)}+\hatS^{\top}\hatK^{(i)}+\hatK^{(i)\top}\hatS+\hatQ=0.
\end{aligned}
\end{equation}
Combining Lyapunov Recursion with Policy improvement \eqref{algorithm 1-improvement}-\eqref{algorithm 1-improvement-2}, we construct a new policy iteration called {\bf Lyapunov recursion scheme}. This scheme requires all the coefficients of the system \eqref{MF-system}. The feasibility and convergence of this algorithm are contained in the following result. 

\begin{theorem}\label{th-2}
Assume that the PDC \eqref{CPD} holds and $(K^{(0)},\hatK^{(0)})$ is a stabilizer for system \eqref{MF-system}. 
Then all the control gains $\{(K^{(i)},\hatK^{(i)})\}_{i=1}^\infty$ in Lyapunov Recursion \eqref{Lyapunov Recursion}-\eqref{Lyapunov Recursion-2} updated by \eqref{algorithm 1-improvement}-\eqref{algorithm 1-improvement-2} are stabilizers, and a solution $(P^{(i+1)},\hatP^{(i+1)})\in (\cS_{++}^n)^2$ to Lyapunov Recursion \eqref{Lyapunov Recursion}-\eqref{Lyapunov Recursion-2} exists and is unique in each step. 
\end{theorem}

\begin{proof}
We prove by mathematical induction. Since $(K^{(0)},\hatK^{(0)})$ is a stabilizer, by Proposition \ref{lemma-stabilizer}, there exists a unique solution $(P^{(1)},\hatP^{(1)})\in(\cS_{++}^n)^2$ for Lyapunov Recursion \eqref{Lyapunov Recursion}-\eqref{Lyapunov Recursion-2}.

Suppose $i\geq 1$, $(K^{(i-1)},\hatK^{(i-1)})$ is a stabilizer and $(P^{(i)},\hatP^{(i)})\in(\cS^n_{++})^2$ is the unique solution to Lyapunov recursion \eqref{Lyapunov Recursion}-\eqref{Lyapunov Recursion-2}.
We now show $(K^{(i)},\hatK^{(i)})$ in the form of \eqref{KK} with $(P^{(i)},\hatP^{(i)})$ is a stabilizer and a solution $(P^{(i+1)},\hatP^{(i+1)})\in(\cS^n_{++})^2$ to \eqref{Lyapunov Recursion}-\eqref{Lyapunov Recursion-2} exists and is unique.

From Theorem 2.1 in \cite{Li-Li-Peng-Xu-2022}, 
\begin{equation}\label{stabilizer1}
\begin{aligned}
&(A+BK^{(i)})^{\top} P^{(i)}+P^{(i)}(A+BK^{(i)})\\&~~~~~~+(C+DK^{(i)})^{\top} P^{(i)}(C+DK^{(i)})<0. 
\end{aligned}
\end{equation}
So by Proposition \ref{lemma-stabilizer}, Lyapunov Recursion \eqref{Lyapunov Recursion} admits a unique solution $P^{(i+1)}\in\cS^n_{++}$.

Next, we prove that $\hatK^{(i)}$ satisfies the second inequality of \eqref{stabilizable-condition}. By some calculations, since $\hatQ-\hatS^{\top}\hatR^{-1}\hatS>0$, we obtain
\begin{equation}\label{stabilizer2}
\begin{aligned}
&\quad~(\hatA+\hatB\hatK^{(i)})^{\top}\hatP^{(i)}+\hatP^{(i)}(\hatA+\hatB\hatK^{(i)})\\
&=-\big[\hatQ-\hatS^{\top}\hatR^{-1}\hatS+(\hatK^{(i)}+\hatR^{-1}\hatS)^{\top}\hatR(\hatK^{(i)}+\hatR^{-1}\hatS)\big]\\
&\quad\quad-(\hatK^{(i-1)}-\hatK^{(i)})^{\top} (\hatR+\hatD^{\top} P^{(i)}\hatD)(\hatK^{(i-1)}-\hatK^{(i)})\\
&~~~-(\hatC+\hatD\hatK^{(i)})^{\top} P^{(i)}(\hatC+\hatD\hatK^{(i)})<0.\\
\end{aligned}
\end{equation} 
By \eqref{stabilizer1}, \eqref{stabilizer2} and Proposition \ref{lemma-stabilizer}, we conclude that $(K^{(i)},\hatK^{(i)})$ is a stabilizer. Moreover, since $\hatK^{(i)\top}\hatR\hatK^{(i)}+\hatS^{\top}\hatK^{(i)}+\hatK^{(i)\top}\hatS+\hatQ>0$ and $(\hatC+\hatD\hatK^{(i)})^{\top} P^{(i)}(\hatC+\hatD\hatK^{(i)})\geq0$, by Proposition \ref{lemma-stabilizer} again, 
the Lyapunov Recursion \eqref{Lyapunov Recursion-2} admits a unique solution $\hatP^{(i+1)}\in\cS^n_{++}$. This completes the proof.

\end{proof}

\begin{theorem}\label{theorem 3}
The iteration $\{(P^{(i)},\hatP^{(i)})\}_{i=1}^{\infty}$ of Lyapunov recursion scheme converges to the unique solution $(P,\hatP)\in(\cS^n_{++})^2$ of the GAREs \eqref{Riccati_Eq1}-\eqref{Riccati_Eq2}. 
\end{theorem}

\begin{proof}
By Theorem 2.2 in \cite{Li-Li-Peng-Xu-2022}, $P^{(i)}\geq P^{(i+1)}\geq0$ for $i=1,2,\cdots$, and $\{P^{(i)}\}_{i=1}^{\infty}$ converges to the unique solution $P\in\cS^n_{++}$ of the GARE \eqref{Riccati_Eq1}. 

Next, we prove that $\{\hatP^{(i)}\}_{i=1}^{\infty}$ converges to the unique solution of the GARE \eqref{Riccati_Eq2}. Assume $\hatP^{(i)}$ and $\hatP^{(i+1)}$ satisfy Lyapunov Recursion \eqref{Lyapunov Recursion-2} and denote
$\Delta\hatP^{(i+1)}=\hatP^{(i)}-\hatP^{(i+1)}$ and $\Delta\hatK^{(i)}=\hatK^{(i-1)}-\hatK^{(i)}$, by some calculations, we get 
\begin{equation}\label{Delta}
\begin{aligned}
&(\hatA+\hatB\hatK^{(i)})^{\top}\Delta\hatP^{(i+1)}+\Delta\hatP^{(i+1)}(\hatA+\hatB\hatK^{(i)})\\
&~~~+(\hatC+\hatD\hatK^{(i)})^{\top} \Delta P^{(i+1)}(\hatC+\hatD\hatK^{(i)})\\
&~~~+\Delta\hatK^{(i)^{\top}}(\hatR+\hatD^{\top} P^{(i)}\hatD)\Delta\hatK^{(i)}=0.
\end{aligned}
\end{equation}
Since $\hatK^{(i)}$ is a component of a stabilizer of the system \eqref{MF-system}, 
$(\hatC+\hatD\hatK^{(i)})^{\top} \Delta P^{(i+1)}(\hatC+\hatD\hatK^{(i)})\geq0$ and $\Delta\hatK^{(i)^{\top}}(\hatR+\hatD^{\top} P^{(i)}\hatD)\Delta\hatK^{(i)}\geq 0$, 
Lyapunov equation \eqref{Delta} admits a unique solution $\Delta\hatP^{(i+1)}\geq 0$ by Proposition \ref{lemma-stabilizer}. Therefore, $\{\hatP^{(i)}\}_{i=1}^{\infty}$ is monotonically decreasing. Notice $\hatP^{(i)}> 0$, so $\{\hatP^{(i)}\}_{i=1}^{\infty}$ converges to some $\hatP\geq 0$. When $i\rightarrow\infty$, $$\hatK^{(i)}=(\hatR+\hatD^{\top} P^{(i+1)}\hatD)^{-1}(\hatB^{\top}\hatP^{(i+1)}+\hatD^{\top} P^{(i+1)}\hatC+\hatS)$$
converges to $\hatK$ in the form of \eqref{KK}. By some calculations, we confirm that $\hatP\in\cS^n_{++}$ is the unique solution of \eqref{Riccati_Eq2}. The proof is completed. 

\end{proof}

Theorem \ref{th-2} and Theorem \ref{theorem 3} theoretically confirm the Lyapunov recursion scheme's stabilizability  and convergence. It also can be solved in implementation by Kronecker product to obtain the {\it explicit} solution $(P^{(i)},\hatP^{(i)})$. However, in practice, we sometimes only know partial coefficients and observe the state trajectories, which also causes difficulties in solving $(P^{(i)},\hatP^{(i)})$ by the Lyapunov recursion scheme directly. Moreover, in Lyapunov recursion scheme \eqref{Lyapunov Recursion-2}, the calculation of $\hatP^{(i+1)}$ depends on $P^{(i+1)}$ while $P^{(i+1)}$ and $\hatP^{(i+1)}$ in \eqref{algorithm 1-evaluation-1}-\eqref{algorithm 1-evaluation-2} of Algorithm \ref{algorithm} are calculated independently. 

Based on Theorem \ref{th-2} and Theorem \ref{theorem 3}, we establish the stabilizability and convergence of Algorithm \ref{algorithm}. 

\begin{theorem}
Assume that the PDC \eqref{CPD} holds and $(K^{(0)},\hatK^{(0)})$ is a stabilizer for the system \eqref{MF-system}. If $\widehat C+\widehat D\widehat K^{(i)}$ is invertible, then Policy Evaluation \eqref{algorithm 1-evaluation-1}-\eqref{algorithm 1-evaluation-2} admits a unique solution $(P^{(i+1)},\hatP^{(i+1)})\in (\cS_{++}^n)^2$. Moreover, $\{(P^{(i)},\hatP^{(i)})\}_{i=1}^\infty$ converges to the unique solution to GAREs \eqref{Riccati_Eq1}-\eqref{Riccati_Eq2} and all the control gains $\{(K^{(i)},\hatK^{(i)})\}_{i=1}^\infty$ in the form of \eqref{algorithm 1-improvement}-\eqref{algorithm 1-improvement-2} are stabilizers. \end{theorem}

\begin{proof}
We need to prove that solving Policy Evaluation \eqref{algorithm 1-evaluation-1}-\eqref{algorithm 1-evaluation-2} in Algorithm \ref{algorithm} is equivalent to solving the Lyapunov Recursion \eqref{Lyapunov Recursion}-\eqref{Lyapunov Recursion-2}. 

Firstly, suppose $(K^{(i)},\hatK^{(i)})$ is a stabilizer for system \eqref{MF-system}. 
Under the PDC \eqref{CPD}, we have
$K^{(i)\top} RK^{(i)}+S^{\top} K^{(i)}+K^{(i)\top} S+Q>0$. 
By Proposition \ref{lemma-stabilizer}, Lyapunov Recursion \eqref{Lyapunov Recursion}-\eqref{Lyapunov Recursion-2} admits the unique solution $(P^{(i+1)},\hatP^{(i+1)})\in(\cS^n_{++})^2$. 

Taking $K=K^{(i)}$ and $\hatK=\hatK^{(i)}$ in \eqref{MF-systemK} and \eqref{MF-Ex}, applying It\^o's formula to $\big\langle P^{(i+1)}(X^{(i)}-{\E_t}[X^{(i)}]), X^{(i)}-{\E_t}[X^{(i)}]\big\rangle$, then integrating on $[t,\infty)$ and taking conditional expectation ${\E_t}[\cdot]$ on both sides, we obtain 
\begin{align}
&0={\E_t}\int_t^\infty\bigg\{\Big\langle\big((A+BK^{(i)})^{\top} P^{(i+1)}+P^{(i+1)}(A+BK^{(i)})\nonumber\\
&~~~~~+(C+DK^{(i)})^{\top} P^{(i+1)}(C+DK^{(i)})\big)\nonumber\\
&~~~~~~~~~\times(X^{(i)}(s)-{\E_t}[X^{(i)}(s)]), X^{(i)}(s)-{\E_t}[X^{(i)}(s)]\Big\rangle\label{integrating-1}
\end{align}
\[
+\Big\langle(\hatC+\hatD\hatK^{(i)})^{\top} P^{(i+1)}(\hatC+\hatD\hatK^{(i)}){\E_t}[X^{(i)}],{\E_t}[X^{(i)}]\Big\rangle\bigg\}\ds.
\]
Since $P^{(i+1)}\in\cS^n_{++}$ is the unique solution of Lyapunov Recursion \eqref{Lyapunov Recursion}, then \eqref{integrating-1} confirms Policy Evaluation \eqref{algorithm 1-evaluation-1}.

Let $\hatK=\hatK^{(i)}$ in \eqref{MF-Ex}, differentiating $\big\langle\hatP^{(i+1)}{\E_t}[X^{(i)}], {\E_t}[X^{(i)}]\big\rangle$ and integrating it from $t$ to $\infty$ yields
\begin{align}
&~~~~-\big\langle\hatP^{(i+1)}x,x\big\rangle\nonumber\\
&=\int_t^{\infty}\Big\langle\big((\hatA+\hatB\hatK^{(i)})^{\top}\hatP^{(i+1)}+\hatP^{(i+1)}(\hatA+\hatB\hatK^{(i)})\big)\nonumber\\
&~~~~\qquad\qquad\qquad\qquad\times{\E_t}[X^{(i)}(s)],{\E_t}[X^{(i)}(s)]\Big\rangle\ds.\label{integrating-2}
\end{align}
Since $\hatP^{(i+1)}\in\cS^n_{++}$ is the unique solution of Lyapunov Recursion \eqref{Lyapunov Recursion-2}, then \eqref{integrating-2} confirms 
\[
\begin{aligned}\label{Eqi}
&\big\langle\hatP^{(i+1)}x,x\big\rangle=\int_t^{\infty}\Big\langle\big(\hatQ+2\hatS^{\top}\hatK^{(i)}+\hatK^{(i)\top}\hatR\hatK^{(i)}\\\nonumber&\!+(\hatC\!+\!\hatD\hatK^{(i)})^{\top} \!\!P^{(i+1)}(\hatC\!+\!\hatD\hatK^{(i)})\big){\E_t}[X^{(i)}(s)],{\E_t}[X^{(i)}(s)]\Big\rangle\ds.\nonumber
\end{aligned}
\]
Combining with \eqref{algorithm 1-evaluation-1}, Policy Evaluation \eqref{algorithm 1-evaluation-2} is confirmed. Also, the unique solution $(P^{(i+1)},\hatP^{(i+1)})\in(\cS^n_{++})^2$ of Lyapunov Recursion \eqref{Lyapunov Recursion}-\eqref{Lyapunov Recursion-2} satisfies \eqref{algorithm 1-evaluation-1}-\eqref{algorithm 1-evaluation-2}, which implies that existence of solution to \eqref{algorithm 1-evaluation-1}-\eqref{algorithm 1-evaluation-2}.

Next, we prove that the solution $(P^{(i+1)},\hatP^{(i+1)})$ of \eqref{algorithm 1-evaluation-1}-\eqref{algorithm 1-evaluation-2} is unique. Suppose there exists another solution $
\widetilde P^{(i+1)}$
to \eqref{algorithm 1-evaluation-1}, then 
\begin{equation}
\begin{aligned}
&\int_t^{\infty}({\E_t}[X^{(i)}(s)])^\top(\hatC+\hatD\hatK^{(i)}) ^\top(P^{(i+1)}-\widetilde P^{(i+1)})\\
&\qquad\qquad\qquad\qquad\qquad\times (\hatC+\hatD\hatK^{(i)}){\E_t}[X^{(i)}(s)]\ds=0.
\label{unique-1}
\end{aligned}
\end{equation}
Taking $\hatK=\widehat K^{(i)}$ in \eqref{MF-Ex}, we obtain $\E_t[X^{(i)}(s)]=e^{(\hatA+\hatB\hatK^{(i)})(s-t)}x$ and insert it into \eqref{unique-1}, since $x$ is chosen randomly and $e^{-(\hatA+\hatB\hatK^{(i)})t}$ is invertible, then
\[
\begin{aligned}
	&\int_t^{\infty}(e^{(\hatA+\hatB\hatK^{(i)})s})^\top(\hatC+\hatD\hatK^{(i)})^\top(P^{(i+1)}-\widetilde P^{(i+1)})
\\&\qquad\qquad\qquad\qquad\qquad\times (\hatC+\hatD\hatK^{(i)})e^{(\hatA+\hatB\hatK^{(i)})s}\ds=0.
\end{aligned}
\]
Taking the derivative of $t$, we get $P^{(i+1)}-\widetilde P^{(i+1)}=0$ since $e^{-(\hatA+\hatB\hatK^{(i)})s}$ and $\widehat C+\widehat D\widehat K^{(i)}$ are invertible. So the solution to \eqref{algorithm 1-evaluation-1} is unique. The unique solvability of \eqref{algorithm 1-evaluation-2} can be proved similarly.

Because Lyapunov Recursion \eqref{Lyapunov Recursion}-\eqref{Lyapunov Recursion-2}  admits a unique solution satisfying  Policy Evaluation \eqref{algorithm 1-evaluation-1}-\eqref{algorithm 1-evaluation-2} and the solution of \eqref{algorithm 1-evaluation-1}-\eqref{algorithm 1-evaluation-2} is unique, we conclude that \eqref{Lyapunov Recursion}-\eqref{Lyapunov Recursion-2} and \eqref{algorithm 1-evaluation-1}-\eqref{algorithm 1-evaluation-2} are equivalent.  

From Theorems \ref{th-2}-\ref{theorem 3}, the assertion of this theorem is confirmed. 
\end{proof}

\section{Implementation of RL Algorithm on Infinite Horizon}\label{Sec4}

%In this section, we implement Algorithm \ref{algorithm} to solve the problem \eqref{MF-system}. 

\subsection{Vectorization and Kronecker product theory}
In order to implement Algorithm \ref{algorithm}, we need to solve $P^{(i+1)}$ and $\hatP^{(i+1)}$ from \eqref{algorithm 1-evaluation-1}-\eqref{algorithm 1-evaluation-2}.
To overcome this critical difficulty, we adopt vectorization method and Kronecker product theory; see \cite{Murray-Cox-Lendaris-Saeks-2002} for details. 
The approach is explained in detail below. 

Define $\vec(A)$ for $A\in \bbR^{n\times m}$ as a vectorization map from a matrix into an $nm$-dimensional column vector for compact representations, which stacks the columns of $A$ on top of one another. 
For $P\in\cS^{n}$, we define an operator $\vec^+(P)$, which maps $P$ into an $\frac{n(n+1)}{2}$-dimensional vector by stacking the columns corresponding to the diagonal and lower triangular parts of $P$ on top of one another where the off-diagonal terms of $P$ are double. By \cite{Murray-Cox-Lendaris-Saeks-2002}, there exists a matrix $\cT\in \bbR^{n^2\times \frac{n(n+1)}{2}}$ with $\mathrm{rank}(\cT)=\frac{n(n+1)}{2}$ such that $\vec(P)=\cT\vec^+(P)$ for any $P\in\cS^n$. 

Let $A\otimes B$ be a Kronecker product of matrices $A$ and $B$. If $A$, $B$, and $C$ have appropriate dimensions, then we have $\vec(ABC)=(C^{\top}\otimes A)\vec(B)$. Denoting $\cK(A)=A^{\top}\otimes A^{\top}$.

Since $P^{(i+1)}$ and $\hatP^{(i+1)}$ are symmetric, there are $\frac{n(n+1)}{2}$ independent parameters in both $P^{(i+1)}$ and $\hatP^{(i+1)}$. For solving the $\frac{n(n+1)}{2}$ parameters in $P^{(i+1)}$ and $\hatP^{(i+1)}$ respectively, we need to randomly choose $N\geq\frac{n(n+1)}{2}$ different initial state $x_j\in \bbR^n$ to generate corresponding trajectories $X_j(s)=X(s;t,x_j)$ on horizon $[t,\infty)$ with $j=1,2,\cdots,N$.

To solve $P^{(i+1)}$ and $\hatP^{(i+1)}$, we first rewrite the left sides of \eqref{algorithm 1-evaluation-1}-\eqref{algorithm 1-evaluation-2} in terms of Kronecker product as follows:
{\small\begin{equation*}
	\begin{aligned}
&~~~\int_t^{\infty}\Big\langle P^{(i+1)}(\hatC+\hatD\hatK^{(i)}){\E_t}[X_j^{(i)}(s)],(\hatC+\hatD\hatK^{(i)}){\E_t}[X_j^{(i)}(s)]\Big\rangle\ds\\
&=\int_t^{\infty}\cK\Big((\hatC+\hatD\hatK^{(i)}){\E_t}[X_j^{(i)}(s)]\Big)\ds\cdot\big(\cT\vec^+(P^{(i+1)})\big)\\
&=:\cI\Big((\hatC+\hatD\hatK^{(i)}){\E_t}[X_j^{(i)}(s)]\Big)~\cT\vec^+(P^{(i+1)})
\end{aligned}
\end{equation*}}
and $\big\langle\hatP^{(i+1)}x_j,x_j\big\rangle=\cK(x_j)\cT\vec^+(\hatP^{(i+1)})$.

Then, we reinforce the objective functions $\cJ_0(t,x_j;K^{(i)},X_j^{(i)})$ and $\cJ(t,x_j;K^{(i)},X_j^{(i)})$ with respect to $X_j^{(i)}$ with $j=1,2,\cdots,N$ and $K^{(i)}$. Denote \[\cI_{\mathbf X}=\left[ \begin{array}{cc} \cI\big((\hatC+\hatD\hatK^{(i)}){\E_t}[X_1^{(i)}(s)]\Big) \\ \cI\big((\hatC+\hatD\hatK^{(i)}){\E_t}[X_2^{(i)}(s)]\Big)\\
\vdots\\
\cI\big((\hatC+\hatD\hatK^{(i)}){\E_t}[X_N^{(i)}(s)]\Big) \end{array}
\right],~\cK_{\mathbf x}=\left[ \begin{array}{cc} \cK(x_1) \\ \cK(x_2)\\
\vdots\\
\cK(x_N) \end{array}
\right],\] 
\[\mathbb J_0=\left[\begin{array}{cc} \cJ_0(t,x_1;K^{(i)},X^{(i)}_1)\\ \cJ_0(t,x_2;K^{(i)},X^{(i)}_2)\\
\vdots\\
\cJ_0(t,x_N;K^{(i)},X^{(i)}_N)\end{array}
\right],~
\mathbb J=\left[\begin{array}{cc} \cJ(t,x_1;K^{(i)},X^{(i)}_1)\\ \cJ(t,x_2;K^{(i)},X^{(i)}_2)\\
\vdots\\
\cJ(t,x_N;K^{(i)},X^{(i)}_N)\end{array}
\right].
\]

Similar to Assumption 3.2 in Xu {\em et al.} \cite{Xu-Shen-Huang-2023}, we introduce the following assumption.
\begin{assumption}\label{ass4}
	There exists an $N_0>0$ such that for all $N\geq N_0$, ${\rm rank}(\mathcal I_{\bf X}\mathcal T)=\frac{n(n+1)}{2}$ and ${\rm rank}(\mathcal K_{\bf x}\mathcal T)=\frac{n(n+1)}{2}$.
\end{assumption}

In practice, we derive the conditional expectation $\E_t[X^{(i)}_j(\cdot)]$ by calculating the mean value based on $H$ sample paths $X_{j,h}$ with $h=1,2,\cdots,H$. Precisely, $
\E_t[X_{j}^{(i)}(s)]\approx\frac{1}{H}\sum_{h=1}^{H}X_{j,h}^{(i)}(s).$ Moreover, $\cJ_0(t,x_j;K^{(i)},X_j^{(i)})$ is calculated by $H$ sample paths with the data sampled at times $s_l\geq0$ with $l=1,2, \cdots, L$, where $L$ is large enough,
\[
\begin{aligned}
&~~~\cJ_0(t,x_j;K^{(i)},X_j^{(i)})\\&\approx\frac{1}{H}\sum_{h=1}^{H}\Big[\sum_{l=1}^L \Big\langle(Q+2SK^{(i)}+K^{(i)\top} RK^{(i)})(X_{j,h}^{(i)}(s_l)\\&\qquad\quad-\frac{1}{H}\sum_{h=1}^{H}X_{j,h}^{(i)}(s_l)]),X_{j,h}^{(i)}(s_l)-\frac{1}{H}\sum_{h=1}^{H}X_{j,h}^{(i)}(s_l)]\Big\rangle\Big].\end{aligned}\]
Moreover, $\cJ(t,x_j;K^{(i)},X_j^{(i)})$ and $\mathcal{I}_{\bf X}$ can also be computed using samples similar to $\mathcal{J}_0(t,x_j;K^{(i)},X_j^{(i)})$. 

If $N$ is large enough, we can collect enough trajectories such that Assumption \ref{ass4} is satisfied. Then $(\cI_{\mathbf X}\cT)^\top \cI_{\mathbf X}\cT$ and $(\mathcal K_{\mathbf x}\cT)^\top\mathcal K_{\mathbf x}\cT$ have inverse matrices so that $
\left(\cI_{\mathbf X}\cT\right)\vec^+(P^{(i+1)})=\mathbb J_0$
admits a unique solution 
\begin{equation}\label{N-solution-P}
\vec^+(P^{(i+1)})=\big[\left(\cI_{\mathbf X}\cT\right)^\top \cI_{\mathbf X}\cT\big]^{-1}\left(\cI_{\mathbf X}\cT\right)^\top\mathbb J_0.
\end{equation}
Similarly, $(\cK_{\mathbf x}\cT)\vec^+(\hatP^{(i+1)})=\mathbb J$
admits a unique solution
\begin{equation}\label{N-solution-hatP}
	\vec^+(\hatP^{(i+1)})=\big[(\cK_{\mathbf x}\cT)^\top(\cK_{\mathbf x}\cT)\big]^{-1}(\cK_{\mathbf x}\cT)^\top\mathbb J.
\end{equation}

Finally, we can obtain $P^{(i+1)}$ and $\hatP^{(i+1)}$ by taking the inverse map of $\vec^+(\cdot)$.

\subsection{A Numerical Example}
We are ready to implement Algorithm \ref{algorithm} to solve system \eqref{MF-system}. For comparison, we calculate the numerical example with $n=5$ and $m=2$ at the initial time $t=0$ in \cite{Huang-Li-Yong-2015} by Algorithm \ref{algorithm}. To save space, the coefficients in the system \eqref{MF-system} and cost functional \eqref{Sec2_Cost} are cited from \cite{Huang-Li-Yong-2015} directly. It is worth pointing out that \cite{Huang-Li-Yong-2015} used all of the coefficients in system \eqref{MF-system} to calculate the solution $(P,\hatP)$. By contrast, we calculate $(P,\hatP)$ without knowing $A$ and $\hatA$. 

We randomly chose more than $\frac{5\times 6}{2}=15$ initial state values according to the uniform distribution on $[0,20]$ such that Assumption \ref{ass4} is satisfied. Then we run system \eqref{MF-system} with $(K^{(0)},\hatK^{(0)})$ by observing the trend of state trajectories when $t$ grows to find an initial stabilizer. Here, we use the Monte Carlo method to simulate the trajectories of the original state with $(K^{(0)},\widehat K^{(0)})$ as follows
\begin{equation*}
\left\{
\begin{aligned}
&X(s+\Delta s)\\&=X(s)+\Big[\big(A+BK\big)\big(X(s)-\bar X(s)\big)+\big(\widehat A+\widehat B\widehat K\big)\bar X(s)\Big]\Delta s \\
& ~~~+\Big\{ \big(C+DK\big)\big(X(s)-\bar X(s)\big)+\big(\widehat C+\widehat D\widehat K\big)\bar X(s)\Big\} \Delta W(s),\\
& \bar X(s+\Delta s) =\bar X(s)+\Big(\widehat A+\widehat B\widehat K\Big)\bar X(s)\Delta s,\quad s\in [t,\infty),
\end{aligned}
\right.
\end{equation*}
where the time interval $\Delta s=0.01$ and $\Delta W(s)= Z\sqrt{\Delta s}$ with $Z$ being the standard normal distribution. 

\begin{figure*}[htbp]
\centering

\begin{minipage}[b]{1\textwidth}
\subfigure[]{\includegraphics[width=0.33\textwidth]{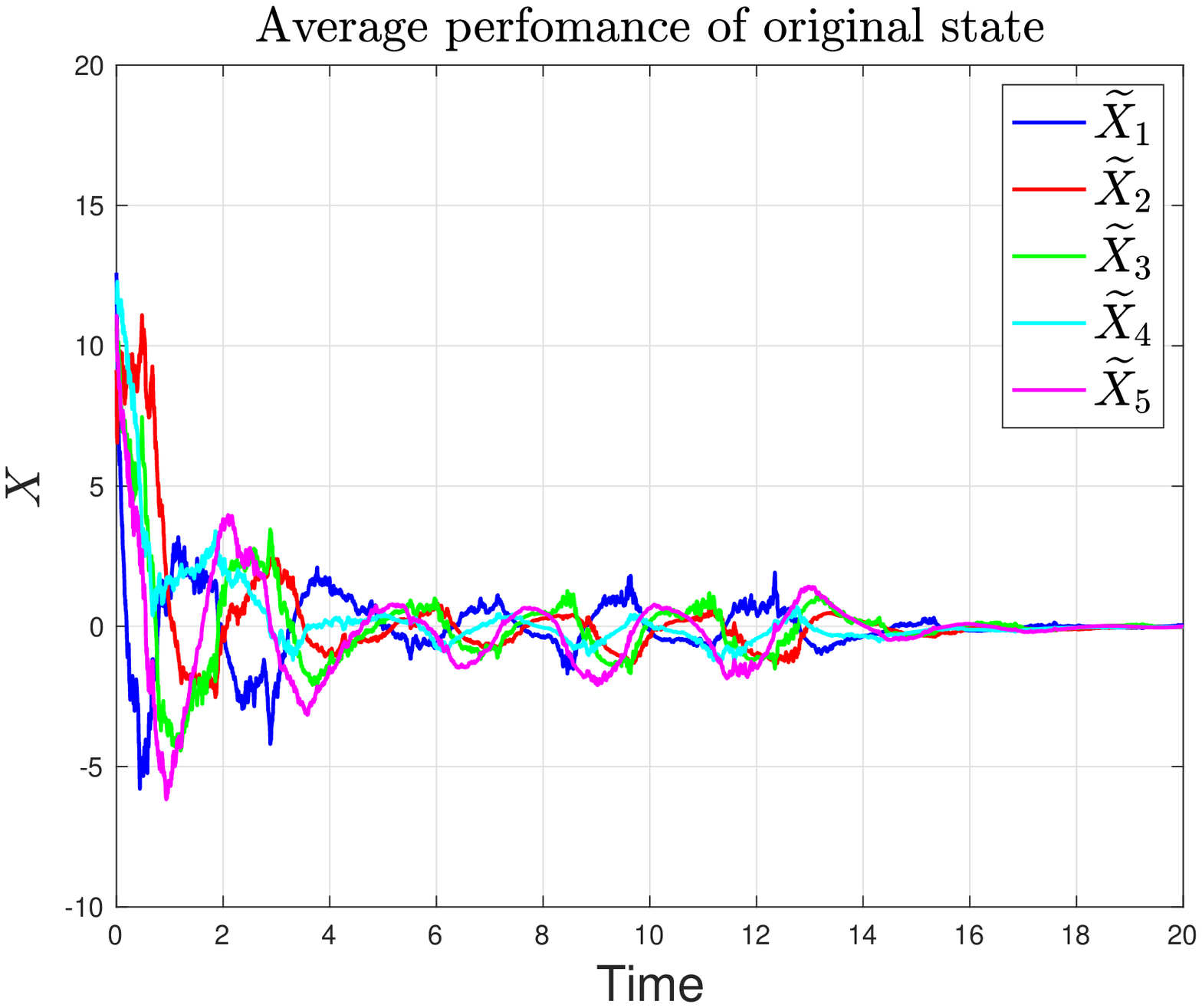}}\hspace{0.02in}
\subfigure[]{\includegraphics[width=0.33\textwidth]{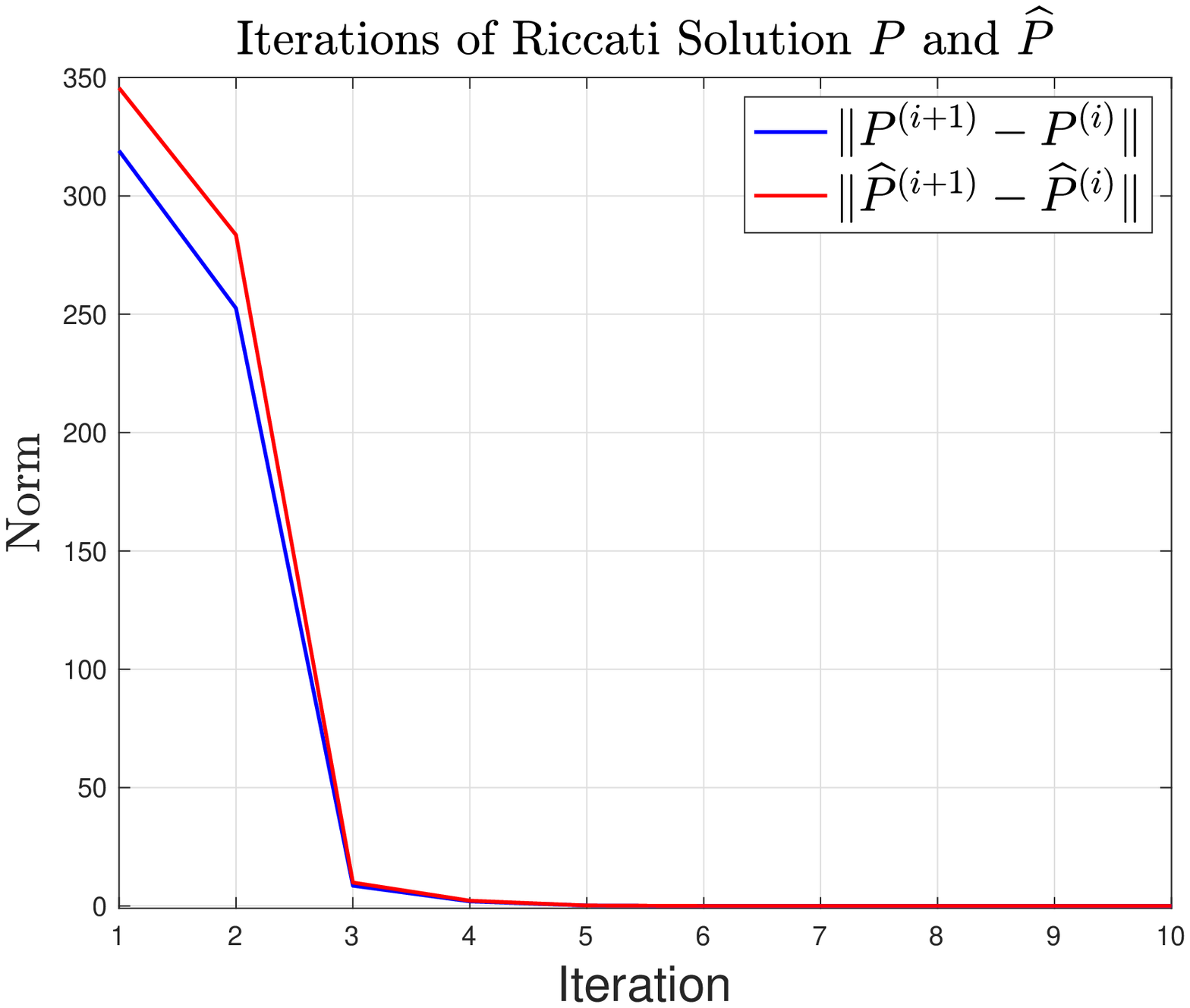} }\vspace{0.02in}
\subfigure[]{\includegraphics[width=0.33\textwidth]{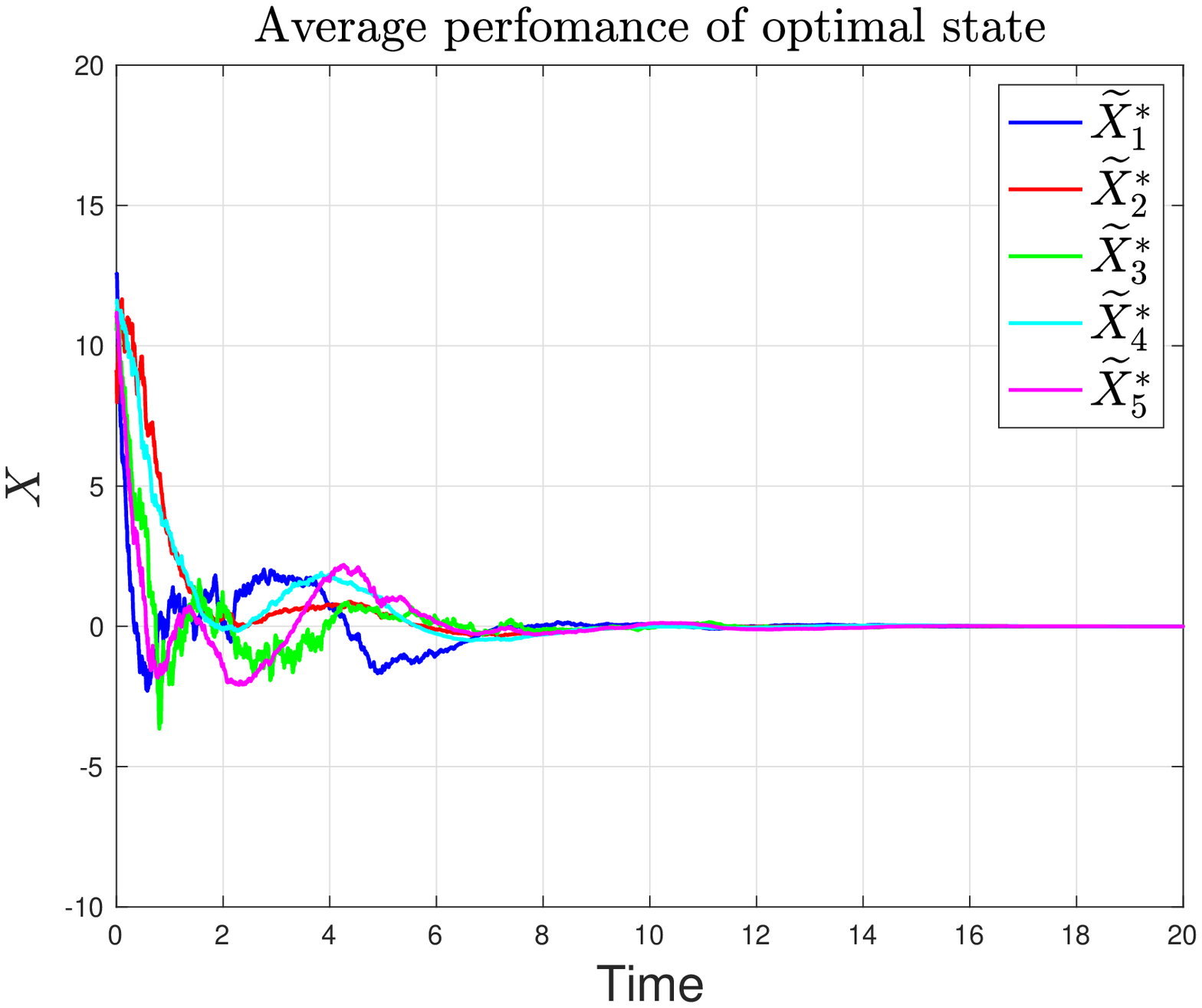} }\vspace{0.02in} 
\end{minipage}
\caption{Simulation results for solutions. (a): The mean value of $15$ original state trajectories of $X$ running with the initial stabilizer $(K^{(0)},\hatK^{(0)})$; (b): The variation of the differences between $P^{(i+1)}$ and $P^{(i)}$, $\widehat P^{(i+1)}$ and $\widehat P^{(i)}$ in each iteration; (c): The mean value of the $15$ pre-commitment optimal state trajectories of $X$ running with the pre-commitment optimal control $u^*$.}
\label{fig:subfig} 
\end{figure*}

We find
\[
K^{(0)}=\begin{bmatrix} 
-1 & 1 & 1 & 1 & -2 \\
1 & 1 & 1 & -1 & -1
\end{bmatrix},~\hatK^{(0)}=\begin{bmatrix} 
-1 & 1 & 1 & -1 & -1 \\
1 & 1 & 1 & -1 & -1
\end{bmatrix}
\] 
can make the trajectories tend to zero when $t$ grows.
Therefore, we choose them as the initial stabilizer. 
Fig. 1 (a) shows the mean value of the $15$ state trajectories running by $(K^{(0)},\widehat K^{(0)})$ in system \eqref{MF-systemK} with different initial states presented above.

Here, we set $H=100000$ and $L=\frac{20}{0.01}=2000$ in this example. Following Algorithm \ref{algorithm}, we update the policy to reinforce the objective functions to obtain $(P,\hatP)$. From \eqref{N-solution-P} and \eqref{N-solution-hatP}, we obtain 
{\small\[
P=\left[\begin{array}{rrrrr}
0.41512 & 0.38898 & 0.20676 & 0.01616 & -0.40586 \\
0.38898 & 2.72081 & 1.90975 & -2.60739 & -0.77561 \\
0.20676 & 1.90975 & 1.85354 & -1.83304 & -0.89785 \\
0.01616 & -2.60739 & -1.83304 & 4.24034 & -0.26645 \\
-0.40586 & -0.77561 & -0.89785 & -0.26645 & 2.15374 \\
\end{array}\right],
\] }
{\small\[
\hatP=\left[\begin{array}{rrrrr}
0.61469 & 0.57206 & 0.26440 & -0.14555 & -0.61375 \\
0.57206 & 4.25787 & 2.87065 & -4.41580 & -0.65355 \\
0.26440 & 2.87065 & 2.67583 & -2.66534 & -1.08896 \\
-0.14555 & -4.41580 & -2.66534 & 6.81576 & -1.06741 \\
-0.61375 & -0.65355 & -1.08896 & -1.06741 & 3.16407 \\
\end{array}\right]
\] 
using 11 iterations for $P$ and $\hatP$. In Fig. 1 (b), we present values of $\|P^{(i+1)}-P^{(i)}\|$ and $\|\widehat P^{(i+1)}-\widehat P^{(i)}\|$ in each iteration to illustrate the variation of the differences of $P^{(i+1)}$ and $P^{(i)}$, $\widehat P^{(i+1)}$ and $\widehat P^{(i)}$.
%Please see the following iteration figures of $P$ and $\widehat{P}$. At the first and second iterations, the values of $P$ and $\widehat{P}$ are so large that they exceed the range of the figures, so we present the third iteration to 12th iteration in Figure \ref{FP}. 

To check whether $(P,\hatP)$ is the solution of GAREs, we define the left sides of \eqref{Riccati_Eq1} and \eqref{Riccati_Eq2} as ${\cR}(P)$ and $\widehat {\cR}(P,\hatP)$, respectively. Inserting $(P,\hatP)$ into them, we obtain $\|\cR(P)\|=3.2474\times10^{-5}$ and $\|\widehat {\cR}(P,\hatP)\|=8.7015\times10^{-5}$, which means that the solution has the high accuracy. In contrast to the result in Huang {\it et al.} \cite{Huang-Li-Yong-2015}, our algorithm is implemented without the information of $A$ and $\hatA$. 
Also, the pre-commitment optimal control is $u^*=K\big(X^* -{\E_t} [X^*] \big)
+\hatK {\E_t} [X^*]$ with 

\[
K=\begin{bmatrix}
-0.35740 & 0.02276 & 0.16816 & -0.11020 & -0.13008 \\
0.21879 & 0.60719 & 0.67927 & -0.77177 & -0.31577 
\end{bmatrix},
\]
\[
\hatK=\begin{bmatrix}
-0.37378 & 0.10978 & 0.28080 & -0.01921 & -0.17809 \\
0.19252 & 0.47965 & 0.60263 & -0.62934 & -0.29326 
\end{bmatrix}.
\]
Similar to the simulation of the original state trajectories, the mean value of the $15$ pre-commitment optimal trajectories running by $(K,\hatK)$ in the system \eqref{MF-systemK} is shown in Fig. 1 (c). Comparing to the original state with $(K^{(0)},\hatK^{(0)})$, we can see that the pre-commitment optimal state converges to zero more quickly than the original state in Fig. 1 (a).

\section*{Acknowledgment}

The authors would like to thank Prof. Jongmin Yong for many helpful discussions and suggestions.
They also thank the associate editor and anonymous referees for their valuable comments and suggestions for improving the current version of the paper. 

\par


\begin{thebibliography}{00}

\bibitem{Rami-Zhou-2000}
M. Ait Rami and X. Y. Zhou, ``Linear matrix inequalities, Riccati equations, and indefinite stochastic linear quadratic controls", {\em IEEE Trans. Automat. Contr.}, vol. 45, pp. 1131-1143, 2000.

%\bibitem{Borkar-Kumar-2010}
%V. S. Borkar and K. S. Kumar, ``McKean-Vlasov limit in portfolio optimization," {\em Stochastic Analysis and Applications}, vol. 28, pp. 884-906, 2010.

%
%\bibitem{Baird-1994}
%L. C. Baird, ``Reinforcement learning in continuous time: Advantage updating," in Proc. IEEE Int. Conf. Neural Netw., pp. 2448-2453, 1994.

\bibitem{Bian-Jiang-Jiang-2016}
T. Bian, Y. Jiang and Z. P. Jiang, ``Adaptive dynamic programming for stochastic systems with state and control dependent noise", {\em IEEE Trans. Automat. Contr.}, vol. 61, pp. 4170-4175, 2016.


\bibitem{Bittanti-Laub-Willems-1991}
S. Bittanti, A. J. Laub, and J. C. Willems, {\em The Riccati Equation}. Germany: Springer-Verlag, 1991.


\bibitem{Bismut-1976}
J. M. Bismut, ``Linear quadratic optimal control with random coefficients," {\em SIAM J. Contr. Optim.}, vol. 14, pp. 419-444, 1976.
\bibitem{Bradtke-Ydestie-Barto-1994}
S. J. Bradtke, B. E. Ydestie and A. G. Barto, ``Adaptive linear quadratic control using policy iteration," in Proc. Amer. Control Conf., pp. 3475-3476, 1994.


%\bibitem{Bradtke-Duff-1995}
%S. J. Bradtke and M. O. Duff, ``Reinforcement learning methods for continuous-time Markov decision problems". In G. Tesauro, D. S. Touretzky, and T. K. Leen (Eds.), {\em Advances in neural information processing systems}, vol. 7 pp. 393-400. Cambridge, MA: MIT Press, 1995.

\bibitem{Bensoussan-Frehse-Yam-2013}
A. Bensoussan, J. Frehse and P. Yam, {\em Mean Field Games and Mean Field Type Control Theory}, Springer Briefs in Mathematics, Springer, 2013. 



%\bibitem{Chan-1994}
%T. Chan, ``Dynamics of the McKean-Vlasov equation", {\it Annals of Probability}, vol. 22, pp. 431-441, 1994.

\bibitem{Chen-Qu-Tang-Low-Li-2022}
X. Chen, G. Qu, Y. Tang, S. Low and N. Li, ``Reinforcement Learning for Selective Key
Applications in Power Systems: Recent
Advances and Future Challenges", {\it IEEE Trans. Smart Grid}, vol. 13, pp. 2935-2958, 2022.
% doi: 10.1109/TSG.2022.3154718.

\bibitem{Du-Meng-Zhang-2020}
K. Du, Q. Meng and F. Zhang, ``A Q-learning algorithm for discrete-time linear-quadratic control with random parameters of unknown distribution: Convergence and stabilization", vol. 60, pp.1991-2015, 2022.

\bibitem{Elie et al.-2020}
R. {\'E}lie, J. P{\'e}rolat, M. Lauri{\`e}re, M. Geist, and O. Pietquin, 
``On the Convergence of Model Free Learning in Mean Field Games", {\em Proceedings of the Thirty-Fourth AAAI Conference on Artificial Intelligence}, vol. 34, pp. 7143-7150, 2020.



\bibitem{Gao-Xu-Zhou-2022}
X. Gao, Z. Q. Xu and X. Y. Zhou, ``State-dependent temperature control for Langevin diffusions", {\em SIAM J. Control Optimiz.}, vol. 60, pp. 1250-1268, 2022. 

\bibitem{Huang-Li-Yong-2015}
J. Huang, X. Li and J, Yong, ``A Linear-quadratic optimal control problem for mean-field stochastic differential equations in infinite horizon", {\em Math. Control Relat. F.}, vol. 5, pp. 97-139, 2015.


\bibitem{Huang-Caines-Malhame-2007}
M. Huang, P. E. Caines and Malhamé, R. P., ``Large-population cost-coupled LQG problems with nonuniform agents: Individual-mass behavior and decentralized $\varepsilon$-Nash equilibria", {\em IEEE Transactions on Automatic Control}, vol. 52, pp. 1560-1571, 2007.


%\bibitem{Kac-1956}
%M. Kac, ``Foundations of kinetic theory", {\em Proc. 3rd Berkeley Sympos. Math. Statist. Prob.}, vol. 3, pp. 171-197, 1956.

\bibitem{Lasry-Lions}
J. M. Lasry and P. L. Lions, ``Mean field games", {\em JPN J. Math.}, vol. 2, pp. 229-260, 2007.

\bibitem{Lauriere et al.-2022}
 M. Lauri{\`e}re, S. Perrin, S. Girgin, P. Muller, A. Jain, T. Cabannes, G. Piliouras, J. P{\'e}rolat, R. {\'E}lie, O. Pietquin, and M. Geist,
``Scalable Deep Reinforcement Learning Algorithms for Mean Field Games", {\em Proceedings of the 39th International Conference on Machine Learning}, vol. 162, pp. 12078-12095, 2022. 
%https://arxiv.org/abs/2203.11973. 

\bibitem{Lewis-Vrabie-Vamvoudakis-2012}
F. L. Lewis, D. Vrabie and K. G. Vamvoudakis, ``Reinforcement learning and Feedback Control", {\em IEEE Contr. Syst. Mag.}, vol. 32, pp.76-105, 2012.

\bibitem{Li-Shi-Yong-2021}
X. Li, J. Shi and J. Yong, ``Mean-field linear-quadratic stochastic differential games in an infinite horizon", {\it ESAIM Contr. Optim. Ca.}, vol. 27, pp. 1-39, 2021.

\bibitem{Li-Li-Peng-Xu-2022}
N. Li, X. Li, J. Peng and Z. Q. Xu, ``Stochastic linear quadratic optimal control problem: A reinforcement learning method", {\em IEEE Trans. Automat. Contr.}, vol. 67, pp. 5009-5016, 2022.




%\bibitem{McKean-1966}
%H. P. McKean, ``A class of Markov processes associated with nonlinear parabolic equations",
%{\em Proc. Natl. Acad. Sci. USA}, vol. 56, 1907-1911, 1966.
%
\bibitem{Minsky-1954}
M. L. Minsky, ``Theory of neural-analog reinforcement systems and its application to the brain model problem," Ph.D. dissertation, Princeton University, 1954. 



\bibitem{Murray-Cox-Lendaris-Saeks-2002}
J. J. Murray, C. J. Cox, G. G. Lendaris and R. Saeks, ``Adaptive dynamic
programming", {\em IEEE T. Syst. Man Cy-S}, vol. 32, pp. 140-153, 2002.



\bibitem{Perrin et al.-2021}
S. Perrin, M. Lauri{\`e}re, J. P{\'e}rolat, M. Geist, R. {\'E}lie, O. Pietquin.
``Mean Field Games Flock! The Reinforcement Learning Way", {\em Proceedings of the Thirtieth International Joint Conference on Artificial Intelligence}, pp. 356-362, 2021.


\bibitem{Sun-Yong-2018}
J. Sun and J. Yong, ``Stochastic linear quadratic optimal control problems in infinite horizon", {\em Appl. Math. Optim.}, vol. 78, pp. 145-183, 2018.

%\bibitem{Wang-Zhang-Zhang-2020}
%B.C. Wang, H. Zhang, J.F. Zhang, ``Mean field linear–quadratic control: Uniform stabilization and social optimality", {\em Automatica}, vol. 121, 109088, 2020. 



\bibitem{Wang-Zariphopoulou-Zhou-2019}
H. Wang, T. Zariphopoulou and X. Y. Zhou, ``Reinforcement learning in continuous time and space: A stochastic control approach", {\em Journal of Machine Learning Research}, vol. 21, pp. 1-34, 2020.

\bibitem{Wang-Zhou-2019}
H. Wang and X. Y. Zhou, ``Continuous-time mean–variance portfolio selection: A reinforcement learning framework", {\em Mathematical Finance}, vol. 30, pp. 1273-1308, 2020.

\bibitem{Xu-Shen-Huang-2023}
Z. Xu, T. Shen, and M. Huang. ``Model-free policy iteration approach to NCE-based strategy design for linear quadratic Gaussian games", {\em Automatica}, vol. 155, 111162, 2023.

\bibitem{Yong-2013} 
J. Yong, ``Linear-quadratic optimal control problems for mean-field stochastic differential equations", {\em SIAM J. Control Optimiz.}, vol. 51, pp. 2809-2838, 2013.

\bibitem{Yong-2017}
J. Yong, ``Linear-quadratic optimal control problems for mean-field stochastic differential equations-time-consistent solutions", {\em T. Am. Math. Soc.}, vol. 369, pp. 5467-5523, 2017.

\bibitem{Yong-Zhou-1999}
J. Yong and X. Y. Zhou, \sl Stochastic controls: Hamiltonian systems
and HJB equations, \rm Applications of Mathematics (New York), 43,
Springer-Verlag, New York, 1999.


\end{thebibliography}
\end{document}